\documentclass[12pt]{article}
\usepackage{mathrsfs}
\usepackage{natbib,graphicx,setspace,lscape,longtable,xr}
\usepackage{mathrsfs,amsmath,amsthm,amssymb,color}
\usepackage{natbib,epsfig,graphicx,pdfpages}
\usepackage{rotating}
\usepackage{float}
\usepackage{bm}
\usepackage{ulem}
\usepackage{CJK}
\usepackage{ragged2e}
\usepackage{setspace}
\usepackage{threeparttable}
\usepackage{multirow}
\usepackage{enumitem}
\usepackage{hyperref}
\usepackage{etoolbox}
\usepackage{subfig}
\usepackage[ruled]{algorithm2e}
\usepackage{adjustbox}
\usepackage{blindtext}
\usepackage{authblk}
\usepackage{enumerate}
\usepackage{cleveref}
\usepackage{mathtools}
\usepackage{threeparttable}

\bibpunct{(}{)}{;}{a}{,}{,}

\newtheorem{theorem}{Theorem}
\newtheorem{lemma}{Lemma}
\newtheorem{proposition}{Proposition}

\allowdisplaybreaks[4]

\newcounter{Condition}
\newcommand{\mylabel}[2]{#2\def\@currentlabel{#2}\label{#1}}
\newcommand{\csection}[1]
{\begin{center}
		\stepcounter{section}
		{\bf\large\arabic{section}. #1}
	\end{center}
}

\newcommand{\csubsection}[1]{
	\begin{center}
		\stepcounter{subsection}
		{\it\arabic{section}.\arabic{subsection}. #1}
	\end{center}
}

\def\tr{\mbox{tr}}

\def\ve{\varepsilon}

\def\beq{\begin{equation}}
	\def\eeq{\end{equation}}
\def\beqr{\begin{eqnarray}}
	\def\eeqr{\end{eqnarray}}
\def\beqrs{\begin{eqnarray*}}
	\def\eeqrs{\end{eqnarray*}}
\def\bet{\begin{theorem}}
	\def\eet{\end{theorem}}
\def\bel{\begin{lemma}}
	\def\eel{\end{lemma}}
\def\bep{\begin{proposition}}
	\def\eep{\end{proposition}}
\def\bg{\begin{figure}[tbph]\begin{center}}
		\def\eg{\end{center}\end{figure}}

\def\bc{\begin{center}}
	\def\ec{\end{center}}

\newtheorem{corollary}{Corollary}

\def\wt{\widetilde}

\def\wh{\widehat}

\def\mN{\mathcal{N}}

\def\mR{\mathbb{R}}
\def\mS{\mathbb S}
\def\mL{\mathcal L}
\def\mH{\mathcal H}

\def\mS{\mathcal S}

\def\var{\mbox{var}}

\def\cov{\mbox{cov}}
\def\argmin{\mbox{argmin}}

\def\supp{\mbox{supp}}

\usepackage[a4paper,bindingoffset=0.2in,%
left=1in,right=1in,top=1in,bottom=1in,%
footskip=.25in]{geometry}

\def\boxit#1{\vbox{\hrule\hbox{\vrule\kern6pt\vbox{\kern6pt#1\kern6pt}\kern6pt\vrule}\hrule}}

\numberwithin{equation}{section}

\title{\textbf{Profiled Transfer Learning for High Dimensional Linear Model}}

\date{} 

\author[1]{Ziqian Lin}
\author[2]{Junlong Zhao}
\author[3\footnote{Corresponding author. Email: wangfang226@sdu.edu.cn}]{Fang Wang}
\author[1]{Hansheng Wang}

\affil[1]{\textit{Guanghua School of Management, Peking University}}
\affil[2]{\textit{School of Statistics, Beijing Normal University}}
\affil[3]{\textit{Data Science Institute, Shandong University}}

\begin{document}
\begin{CJK}{GBK}{song}


\maketitle

\doublespacing
\begin{singlespace}
	\begin{abstract}
		We develop here a novel transfer learning methodology called Profiled Transfer Learning (PTL). The method is based on the \textit{approximate-linear} assumption between the source and target parameters. Compared with the commonly assumed \textit{vanishing-difference} assumption and \textit{low-rank} assumption in the literature, the \textit{approximate-linear} assumption is more flexible and less stringent. Specifically, the PTL estimator is constructed by two major steps. Firstly, we regress the response on the transferred feature, leading to the profiled responses. Subsequently, we learn the regression relationship between profiled responses and the covariates on the target data. The final estimator is then assembled based on the \textit{approximate-linear} relationship. To theoretically support the PTL estimator, we derive the non-asymptotic upper bound and minimax lower bound. We find that the PTL estimator is minimax optimal under appropriate regularity conditions. Extensive simulation studies are presented to demonstrate the finite sample performance of the new method. A real data example about sentence prediction is also presented with very encouraging results.\\

	\end{abstract}
\end{singlespace}

\clearpage

\csection{INTRODUCTION}

High dimensional data have commonly emerged in many applications \citep{deng2009imagenet, fan2020statistical}. However, one often has very limited sample sizes due to various practical reasons such as the ethical issues, data collection costs, and possibly others \citep{lee2017medical,wang2020generalizing}. Consider for example a lung cancer dataset featured in the work by \cite{zhou2023ensemble}. In this study, 3D computed tomography (CT) images with unified voxel sizes of approximately $60\times 60 \times 60$ were collected for each patient. The objective there was to diagnose the nodule type being benign or malignant. 
However, the cost for diagnosing nodule type is notably high and the sample size can hardly be large. 
In the study of \cite{zhou2023ensemble}, only a total of 1,075 subjects were collected. 
Taking into account the covariates (i.e. CT images) being of size about $60\times 60 \times 60$, the sample size of 1,075 is exceedingly small. Similar issues can also be found in numerous other disciplines, including marketing \citep{wang2009forward}, computer vision \citep{vinyals2016matching,bateni2020improved}, and natural language processing \citep{mcauliffe2007supervised}. 

It is well known that sample sizes play a fundamental role in statistical modeling, determining the accuracy of estimators or powers of test statistics. Small sample sizes raise serious challenges for high dimensional data analysis. One way to address this problem is to impose some low-dimensional structures on the high-dimensional model of interest. For example, when high dimensional linear model are concerned, it is commonly assumed that the parameters are sparse, resulting in variously shrinkage estimation methods such as LASSO, SCAD, MCP and many others \citep{tibshirani1996regression, fan2001variable,zou2005regularization, yuan2006model, zou2006adaptive,wang2007unified, zhang2010nearly} and feature screening methods \citep[etc.]{fan2008sure,wang2009forward,fan2010sure,zhu2011model,li2012feature}. 
Another commonly used approach is to assume some (approximately) low rank structures. For example, reduced rank regression assumes a low-rank coefficient matrix \citep{anderson1951estimating, izenman1975reduced,anderson1999asymptotic, chen2012sparse, reinsel2022multivariate}, while factor model assumes the covariance matrix is approximately low rank \citep{bai2002determining,stock2002forecasting,bai2003inferential,fan2008high,bai2012statistical,lam2012factor,wang2012factor}. However, these assumptions may be violated. Consider for example a study exploring the effect of single-nucleotide polymorphisms (SNPs) on height, as reported by \cite{boyle2017expanded}. In this study, the total sample size amounts to 253,288 with a feature dimension of approximately 10 million. Moreover, this study suggests that there are about 400,000 common SNPs potentially having nonzero causal effects on height. Consequently, the associated regression coefficient vector is not too sparse. Recently estimations of nonsparse parameters have been also considered \citep{azriel2020estimation,silin2022canonical,zhao2023estimation}. Similarly the low dimensional structure is hard to verify and may be violated in practice. 

To solve the problem, a number of researchers have advocated the idea of transfer learning \citep{pan2009survey,weiss2016survey}.  For convenience, the dataset associated with the problem at hand is referred to as the target dataset. The main idea of transfer learning is to borrow strengths from some  auxiliary datasets that will be referred to as source datasets. The parameters associated with sources and the target are termed the source parameters and the target parameters, respectively. To borrow strengths from source datasets, certain types of assumptions have to be imposed on the similarity between the target and source datasets. For example, \cite{li2020transfer} and \cite{zhao2023residual} considered transfer learning of high dimensional linear model, assuming that the $\ell_q$-distance ($0\le q\le 1$) between the target parameter and some of the source parameters should be sufficiently small. For convenience, we refer to this assumption as the \textit{vanishing-difference} assumption. Similar assumption was also used by \cite{tian2022transfer} and \cite{li2023estimation} for generalized linear models,  \cite{lin2022transfer} for functional linear regression, and \cite{cai2022transfer} for nonparametric regression. As an alternative approach, 
\cite{du2020few} and \cite{tripuraneni2021provable} assumed the source parameters and the target parameter can be linearly approximated by a common set of vectors in a low-dimensional subspace. We refer to this as the \textit{low-rank} assumption. When these assumptions are satisfied, it has been shown that statistical efficiency can be enhanced significantly in these works.

Despite the success of the transfer learning methods, the similarity assumptions of these methods (i.e. the \textit{vanishing-difference} assumption and the \textit{low-rank} assumption) are stringent. For example, according to the \textit{vanishing-difference} assumption, only those source datasets with parameters very close to the target one can be used for transfer learning, making the task of detecting transferable source datasets extremely crucial \citep{tian2022transfer}. On the other hand, the \textit{low-rank} assumption is not perfect either. In practice, situations often arise where the parameters associated with  the target and various sources are linearly independent, making it challenging to identify a low rank structure for the linear subspace spanned by these parameters. Even if the \textit{low-rank} assumption holds, how to practically determine the structural dimension with finite data remains an additional challenging issue \citep{bai2002determining,luo2009contour, lam2012factor}.

In this paper, we focus on transfer learning for high dimensional linear models. To implement transfer learning with less stringent similarity assumptions, we propose here a novel method named Profiled Transfer Learning (PTL). Our approach assumes that the target parameter can be linearly approximated by the source parameters without the stringent rank constraint. Specifically, we assume that there are $K$ sources, and consider high dimensional linear models for both target and sources. Denote parameter of the target model as $\bm \beta\in\mR^p$, and the parameter of source $k$ as $\bm \beta^{(k)}\in\mR^p$, for $k = 1,\dots, K$. Then we assume that $\bm \beta = \sum_{k=1}^K w_k \bm \beta^{(k)} + \bm \delta$ for some constants $w_k$'s and $\bm \delta\in\mR^p$. This assumption is referred to as \textit{approximate-linear} assumption.
Unlike the \textit{vanishing-difference} assumption, \textit{approximate-linear} assumption allows the difference between the target and source parameters (i.e. $\|\bm \beta^{(k)} - \bm \beta\|_q$) to be arbitrarily large. Additionally, different from the \textit{low-rank} assumption, the \textit{approximate-linear} assumption allows the structural dimension of the subspace spanned be $\bm \beta^{(k)}$'s to be as large as $K$.

Based on the \textit{approximate-linear} assumption, we propose here a profiled transfer learning method with two major steps. Let $Y\in \mR$ and $X\in \mR^p$ be the response and covariates for the target linear model, respectively. In the first step, we estimate $\wh{\bm \beta}^{(k)}$ using data from source $k$ to generate the transferred features $\wh Z = \wh B^\top X$ in the target data, where $\wh B = (\wh {\bm \beta}^{(k)}, k = 1,\dots, K)\in\mR^{p\times K}$.
Then we regress the response $Y$ on the transferred features $\wh Z$ using  the target data, treating the resulting residuals as the profiled responses. In the second step, the regression relationship between the profiled responses and the covariates $X$ is learned by some  sparse learning method (e.g. LASSO) on the target data. Then the final estimator for the target parameter is assembled.
%
Theoretical and numerical analyses suggest that our estimator can significantly improve the statistical efficiency, even in scenarios where both the \textit{vanishing-difference} and the \textit{low-rank} assumptions are violated.

The remainder of the paper is organized as follows. In Section 2, we introduce the profiled transfer learning method. The non-asymptotic upper bound and the minimax lower bound are presented in Section 3. Simulation experiments and real data analysis are reported in Section 4. We conclude the paper by a brief discussion about future researches in Section 5. All the technical details are provided in the Appendix.

Notations. For any vector $\bm v = (v_1,\dots, v_p)^\top \in \mR^p$, denote $\|\bm v\|_1$, $\|\bm v\|$ and $\|\bm v\|_{\infty}$ as its $\ell_1$, $\ell_2$ and $\ell_{\infty}$ norms, respectively; denote $\|\bm v\|_0 = \sum_{j=1}^p I(v_j\neq 0)$ as the $\ell_0$ norm of $\bm v$ and $\mbox{supp}(\bm v) = \{i:v_j \neq 0,j = 1,\dots, p\}$ as the associated support set. Define $\lambda_{\max}(A)$ and $\lambda_{\min}(A)$ to be the maximum and minimum eigenvalues for an arbitrary symmetric matrix $A\in\mR^{p\times p}$. For any $A\in\mR^{p_1\times p_2}$, we   define its operator norm as $\|A\|_{\text{op}} = \sqrt{\lambda_{\max}(A^\top A)}$ and Frobenius norm as $\|A\|_F = \sqrt{\tr(A^\top A)}$, respectively. For a sub-Gaussian random variable $U\in\mR$, define its sub-Gaussian norm to be $\|U\|_{\psi_2} = \inf\{t>0: E\exp(U^2/t^2) \le 2\}$. For a sub-Gaussian random vector $V$, define its sub-Gaussian norm to be $\|V\|_{\psi_2} = \sup_{\|u\| = 1}\| u^\top V\|_{\psi_2}$. 
Let $a\wedge b = \min\{a,b\}$ and $a\vee b = \max\{a,b\}$ for any real $a,b\in\mR$. Write $a\lesssim b$ if there exists a constant $C>0$ such that $a \le Cb$.

\csection{PROFILED TRANSFER LEARNING FOR HIGH DIMENSIONAL LINEAR MODEL}
\csubsection{Target and Source Models}

Similar to \cite{li2020transfer}, we consider linear regression models for both target and sources in this paper. Let $\{(Y_i,X_i), i=1,\dots, n\}$ be i.i.d. observations from the following target model
\beq
Y_i = X_i^\top \bm \beta + \ve_i, \qquad i=1,\cdots, n, \label{Model target linear}
\eeq
where $Y_i\in\mR$ denotes the response  and $X_i = (X_{i1},\dots, X_{ip})^\top\in\mR^p$ denotes the associated $p$-dimensional covariates, satisfying $E(X_i) = 0$. Denote the covariance matrix of $X_i$ as $\Sigma = \cov(X_i)$. In addition, $\bm \beta = (\beta_1,\dots, \beta_p)^\top \in\mR^p $ is the coefficient vector of the target model and $\ve_i\in\mR$ is the random error independent of $X_i$ satisfying $E(\ve_i) = 0$ and $E(\ve_i^2) = \sigma^2$. Let $\mbox{supp}(\bm\beta) = \{j:\beta_j\neq 0, j = 1,\dots, p\}$ be the support set of $\bm \beta$ with cardinality $s_0$. Here we allow the number of parameters $p$ to be much larger than the sample size $n$.

To estimate the unknown parameter $\bm \beta$, a penalized ordinary least squares can be applied \citep{tibshirani1996regression,fan2001variable} by assuming that $\bm \beta$ is sparse. The loss function is given as
\[
\mL_{\text{lasso},\lambda_\beta}(\bm \beta) = (2n)^{-1}\sum_{i=1}^n \Big(Y_i - X_i^\top \bm \beta\Big)^2 + \lambda_\beta\|\bm \beta\|_1,
\]
where $\lambda_\beta$ is the tuning parameter for $\bm \beta$. Then the LASSO estimator is given as $\wh{\bm \beta}_{\text{lasso}} = \argmin_{\bm \beta} \mL_{\text{lasso},\lambda_\beta}(\bm \beta) $. Under appropriate regularity conditions, one can verify that $\|\wh{\bm \beta}_{\text{lasso}} - \bm \beta\| = O_p(\sqrt{s_0\log p/n})$. See \cite{vershynin2018high} and \cite{wainwright2019high} for some indepth theoretical discussions. However, when the parameter is not too sparse and the sample size $n$ is relatively small in the sense $s_0\log p/n \nrightarrow 0$ as $n\to\infty$, the LASSO estimator might be inconsistent \citep{bellec2018slope}. 

To solve this problem, we turn to transfer learning for help. Specifically, we assume there are a total of $K$ different sources. Denote $\{(X_i^{(k)}, Y_i^{(k)}), i = 1,\dots, n_k\}$ being the i.i.d. observations from the source $k$ for $k=1,\cdots,K$, generated from the following source model
\[
Y_i^{(k)} = X_i^{(k)\top} \bm\beta^{(k)} + \ve_i^{(k)},
\]
where $Y_i^{(k)}\in\mR$ is the response and $X_i^{(k)} = (X_{i1}^{(k)},\dots, X_{ip}^{(k)})^\top \in \mR^p$ is the corresponding covariates satisfying $E(X_i^{(k)}) =0$ and $\cov(X_i^{(k)}) = \Sigma_k$. 
Moreover, $\bm \beta^{(k)} = (\beta_1^{(k)}, \dots, \beta_p^{(k)})^\top\in \mR^p$ is the coefficient vector
and $\ve_i^{(k)}\in\mR$ is the random error independent of $X_i^{(k)}$ with $E(\ve_i^{(k)}) = 0$ and $E(\ve_i^{(k)2}) = \sigma_k^2$. For each source, the parameter can be estimated by some appropriate methods. We denote $\wh {\bm \beta}^{(k)}$ as the estimator for $\bm\beta^{(k)}$ for $k = 1,\dots, K$. If $n_k <p$ and $\bm\beta^{(k)}$ is sparse, then $\wh{\bm\beta}^{(k)}$ can be the LASSO estimator. If the sample size of source data is sufficiently large such that $n_k \gg p$, the classical ordinary least squares (OLS) estimator or ridge regression estimator can be also used, allowing $\bm\beta^{(k)}$ being either sparse or nonsparse.

\csubsection{Profiled Transfer Learning for High Dimensional Linear Regression Model}

We next assume that the $\bm \beta^{(k)}$'s with different $k$'s are linearly independent of each other. In other words, the dimension of the space spanned by $\bm \beta^{(k)}$'s is $K$. This assumption is much weaker than the \textit{low-rank} assumption \citep{du2020few, tripuraneni2021provable}, where it is assumed that $\bm\beta^{(k)}$'s fall into a low-dimensional subspace. To borrow information from the source data, we impose the following  {\it approximate-linear} assumption  
$$
\bm \beta = \sum_{k=1}^{K} w_k \bm \beta^{(k)} + \bm \delta,
$$
 where $w_k \in\mR$ is the weight assigned to $\bm \beta^{(k)}$ and $\bm \delta = (\delta_1,\dots, \delta_p)^\top \in\mR^p$ is the residual part. Ideally, we hope that $\bm \delta$ is small; specifically   $\|\bm\delta\|_1$ is small.
  By this assumption, we allow the differences $\|\bm \beta - \bm \beta^{(k)}\|$ to be large for $1\le k\le K$. Therefore, it is more flexible than the \textit{vanishing-difference} assumption \citep{li2020transfer,tian2022transfer}. 

Let $c_k \in\mR$ be arbitrary constants for $1\le k\le K$. Then we have
\[
\sum_{k=1}^K w_k \bm \beta^{(k)} + \bm \delta = \sum_{k=1}^K (w_k - c_k) \bm \beta^{(k)} + \bm \delta+ \sum_{k=1}^K c_k \bm \beta^{(k)} = \sum_{k=1}^K \bm w_k^*\beta^{(k)} +\bm  \delta^*,
\] 
where $w_k^* = w_k - c_k$ and $\bm \delta^* = \bm \delta + \sum_{k=1}^K c_k\bm \beta^{(k)}$. This suggests that $\bm w = (w_1,\dots, w_K)^\top \in\mR^K$ and $\bm \delta \in \mR^p$ may not be uniquely identified without appropriate constraints. To resolve this problem, we impose the following \textit{identification condition}: $\bm \beta^{(k)\top }\Sigma\bm \delta  = 0$ for every $1\le k\le K$. The following Proposition \ref{Prop identification} shows that both $w_k$'s and $\bm \delta$ are identifiable under this condition.
\begin{proposition}
	Assuming that $\bm \beta^{(k)\top }\Sigma\bm \delta  = 0$ for $1\le k\le K$, the parameters $\bm w$ and $\bm \delta$ are identifiable. \label{Prop identification}
\end{proposition}

The detailed proof of Proposition \ref{Prop identification} is provided in Appendix A. 
Let $B = (\bm \beta^{(1)},\dots, \\\bm \beta^{(K)}) \in\mR^{p\times K}$ and define $Z_i = B^\top X_i, i=1,\dots, n$ as the oracle transferred feature. Then the original linear regression model \eqref{Model target linear} becomes
\[
Y_i  = Z_i^\top \bm w + X_i^\top\bm \delta + \ve_i,i=1,\dots,n.
\]
By the condition $\bm \beta^{(k)\top}\Sigma\bm \delta = 0$ for every $1\le k\le K$, it follows that $E\{Z_i(X_i^\top \bm \delta)\} = 0$. That is, $Z_i$ is uncorrelated with $X_i^\top\bm \delta$.
 This implies that $\bm w$ can be estimated by simply considering the marginal regression of $Y_i$ on $Z_i$.
Inspired by the above discussions, we propose a novel estimation procedure as follows.

Let $\wh B = (\wh{ \bm \beta}^{(1)}, \dots, \wh {\bm \beta}^{(K)}) \in \mR^{p\times K}$ be the estimator of $B$, and define $\wh Z_i = \wh B^\top X_i, i =1,\dots,n$. Then $w$ can be estimated by regressing $Y_i$ on $\wh Z_i$. That is $\wh {\bm w}_{\text{ptl}} = \argmin_{\bm w} \mL_{w}(\bm w; \wh Z)$ with
 \[\mL_{w}(\bm w; \wh Z) = (2n)^{-1}\sum_{i=1}^n \Big(Y_i - \wh Z_i^\top \bm w\Big)^2.\]
 Denote the OLS estimator as $\wh {\bm w}_{\text{ptl}}$, specifically,
 \[
 \wh {\bm w}_{\text{ptl}}= \Big(\sum_{i=1}^n \wh Z_i \wh Z_i^\top \Big)^{-1}\Big(\sum_{i=1}^n \wh Z_i Y_i\Big).
 \]
 This leads to residuals $\wh e_i = Y_i - \wh Z_i^\top \wh {\bm w}_{\text{ptl}}$, which is referred to as a profiled response. Lastly, $\bm \delta$ can be estimated using LASSO by regressing $\wh e_i$'s on $X_i$'s as $\wh{\bm \delta}_{\text{ptl}} = \argmin_{\bm \delta} \mL_{\text{target},\lambda_\delta}(\bm \delta)$, where \[
\mL_{\text{target},\lambda_\delta}(\bm \delta) = (2n)^{-1}\sum_{i=1}^n \Big(\wh e_i - X_i^\top \bm \delta\Big)^2 + \lambda_\delta \|\bm \delta\|_1
\]
with tuning parameter $\lambda_\delta$.  Therefore, the final profiled transfer leaning (PTL)  estimator is given by
$$\wh{\bm \beta}_{\text{ptl}} = \wh B \wh {\bm w}_{\text{ptl}} + \wh{\bm \delta}_{\text{ptl}}.$$
We summarize the whole procedure in the following Algorithm 1.

\begin{algorithm}
	\KwIn{Source datasets $\{(X_i^{(k)}, Y_i^{(k)} ) \}_{i=1}^{n_k}$, $k= 1,\dots, K$ and the target dataset $\{(X_i,Y_i)\}_{i=1}^n$ }
	\KwOut{The PTL estimator $\wh{\bm \beta}_{\text{ptl}}$}
	
	\textit{Step 1. }Compute $
	\wh{\bm \beta}^{(k)} 
	= \argmin_{\bm \beta^{(k)}} (2n_k)^{-1}\sum_{i=1}^{n_k}(Y_i^{(k)} - X_i^{(k)\top}\bm \beta^{(k)})^2 + \lambda_k \|\bm \beta^{(k)}\|_1
$
	with $\lambda_k$ chosen by cross validation. Compute the estimated feature vectors $\wh Z_i = \wh B^\top X_i,i=1,\dots,n$.
	
	\textit{Step 2. }Compute $\wh {\bm w}_{\text{ptl}} = (\sum_{i=1}^n \wh Z_i \wh Z_i^\top)^{-1}(\sum_{i=1}^n \wh Z_i Y_i)$. Compute the profiled responses $\wh e_i = Y_i - \wh Z_i^\top \wh {\bm w}_{ptl},i=1,\dots,n$.
	
	\textit{Step 3.} Compute $\wh{\bm \delta}_{\text{ptl}} = \argmin_{\bm \delta} (2n)^{-1}\sum_{i=1}^n (\wh e_i - X_i^\top \bm \delta)^2 + \lambda_\delta \|\bm \delta\|_1$ with $\lambda_{\delta}$ chosen by cross validation.
	
	\textit{Step 4.} Compute $\wh{\bm \beta}_{\text{ptl}} = \wh B^\top \wh {\bm w}_{\text{ptl}} + \wh{\bm \delta}_{\text{ptl}}$. Output $\wh{\bm \beta}_{\text{ptl}}$.
	\caption{PTL Algorithm}
\end{algorithm}

\csection{STATISTICAL PROPERTIES}

\csubsection{Statistical Consistency}

We next study the statistical consistency of $\wh{\bm \beta}_{\text{ptl}}$. Define $\bm s = (s_{\max},s_{\delta})^\top \in\mR^{2}$. Consider the parameter space
\begin{gather*}
\Theta(\bm s, h) = \Big\{\bm \theta = (B, \beta):
\bm \beta = B\bm w + \bm \delta, \|\bm \beta^{(k)}\|_0 \le s_{\max},\bm \beta^{(k)\top}\Sigma \bm \delta = 0, k =1,\dots, K,\\ \|\bm w\| \le C_w, \|\bm \delta\|_1 \le h, \|\bm \delta\|_0 \le s_\delta \Big\},
\end{gather*}
where $C_w>0$ is a fixed constant. As one can see, the parameters contained in $\Theta(\bm s,h)$ are assumed to be sparse in the sense that $\|\bm\beta^{(k)}\|_0\le s_{\max},k=1,\dots,K$, $\|\bm\delta\|_1\le h$ and $\|\bm\delta\|_0 \le s_{\delta}$.
To establish the statistical properties for the proposed estimator, the following assumptions are needed.

\begin{itemize}
	\item [(C1)] \textsc{(Convergence Rate of Source Parameters)} Assume that for $k=1,\dots, K$, there exists constants $C_k$ and $C_k'$ such that $\|\wh{\bm \beta}^{(k)} - \bm \beta^{(k)}\|_2^2 \le C_k s_{\max}\log p/ n_k$ with probability at least $1-2\exp(-C_k'\log p)$.

	\item [(C2)] \textsc{(Divergence Rate)} Assume that $\sum_{k=1}^K s_{\max}\log p/n_k = o(1)$, $s_{\delta}\log p/n = o(1)$, $K = o(n\wedge p)$,  and $h = O(1)$.
	\item [(C3)] \textsc{(Sub-Gaussian Condition)} Assume that $X_i$ and $\ve_i$ are sub-Gaussian and there exists a positive constant $C_{\psi_2}>0$ such that $\|X_i\|_{\psi_2}\le C_{\psi_2}$ and $\|\ve_i\|_{\psi_2}\le C_{\psi_2}$.
	\item [(C4)] \textsc{(Covariance Matrices)} Assume that there exist absolute constants $0<C_{\min} <C_{\max} <\infty$ such that (1) $C_{\min}\le \lambda_{\min}(\Sigma) \le \lambda_{\max}(\Sigma)\le C_{\max}$, and (2) $C_{\min}\le \lambda_{\min}(B^\top B) \le \lambda_{\max}(B^\top B)\le C_{\max}$.
\end{itemize}

Condition (C1) is mild. Given the sparsity assumption of $\bm \beta^{(k)}$,  $\wh {\bm \beta}^{(k)}$ can be chosen as the LASSO estimator. Under some regularity conditions, one can verify that the statistical error of the LASSO estimator $\wh{\bm{\beta}}^{(k)}$ in terms of $\ell_2$ norm is $O_p\{(s_{\max} \log p/n_k)^{1/2}\}$; see Section 10.6 of \cite{vershynin2018high} and Section 7.3 of \cite{wainwright2019high}.
 Condition (C2) constraints the divergence rates of various quantities. Specifically, the condition $\sum_{k=1}^K s_{\max} \log p/ n_k = o(1)$  ensures that $\wh{\bm \beta}^{(k)}$'s converge uniformly over $k$. In practice, $K$ is generally much smaller than $n$ and $p$, and the assumption $K = o(n\wedge p)$ is considered mild. 
 When the source models are indeed helpful, it is natural to have the profiled target parameter $\bm \delta$ to be sufficiently ``small". As $\|\bm \delta\|_1$ is bounded, the conditions 
 $s_{\delta}\log p/ n = o(1)$ and $h = O(1)$ ensure that $\bm\delta$ can be consistently estimated.

Condition (C3) assumes the covariates and the error terms of the target follow a sub-Gaussian distribution with a uniformly bounded sub-Gaussian norm. This assumption is commonly employed in the literature \citep[etc.]{zhu2018sparse,gold2020inference, tian2022transfer}. 
The eigenvalue condition in (C4) is commonly assumed in the literature \citep[etc.]{cai2018accuracy,li2020transfer}. Based on (C4), it is easy to see that $\lambda_{\max}\{E(Z_iZ_i^\top )\} = \lambda_{\max}(B^\top \Sigma B) \le C_{\max}\lambda_{\max}(B^\top B)\le C_{\max}^2$.

With the help of those technical conditions, we can derive a non-asymptotic upper bound for $\wh{ \bm \beta}_{\text{ptl}}$. For convenience, define \[
r_n= \bigg(\sum_{k=1}^K \frac{s_{\max}\log p}{ n_k}\bigg)^{1/2} + \bigg(\frac Kn\bigg)^{1/2} + \bigg(\frac{\log p}n\bigg)^{1/2},
\] which is the upper bound of the quantity $\| n^{-1}\sum_{i=1}^n X_i(\wh e_i - X_i^\top\bm \delta)\|_{\infty}$. Note that $r_n$ contains three components. The first one is $(\sum_{k=1}^K s_{\max}\log p/ n_k )^{1/2}$, reflecting the statistical error in estimating the source model parameters. The second one is $(K/n)^{1/2}$, representing the statistical error in estimating $\bm w$. The third one is $(\log p/n)^{1/2}$,  which is the rate of $\|n^{-1}\sum_{i=1}^n X_i (e_i - X_i^\top \delta)\|_{\infty}$ where $e_i = Y_i - Z_i^\top w$ is the oracle profiled response when $Z_i$'s and $\bm w$ are known \citep{li2020transfer}. Then we have the following Theorem \ref{Theorem convergence rate}, the detailed proof of which is provided in Appendix C.

\begin{theorem}\label{Theorem convergence rate}
	Assume the conditions (C1)-(C4). Further assume that $\lambda_\delta\ge C_\lambda r_n$ for some sufficiently large constant $C_\lambda >0$. Then for some constant $C_1>0$, it holds that:  
	\beqrs
	\inf_{\theta \in \Theta(\bm s, h)}P\bigg(\|\wh {\bm \beta}_{\textnormal{ptl}} - \bm \beta\|^2 &\lesssim &\sum_{k=1}^K \frac{s_{\max} \log p}{n_k} + \frac{K+\log(p\wedge n)}n \\&+& s_{\delta} r_n^2\wedge hr_n\wedge h^2\bigg) \ge 1- \exp\{-C_1\log (p\wedge n)\}.
	\eeqrs
\end{theorem}
 By Theorem \ref{Theorem convergence rate}, we know that the estimation error as measured by $\|\wh{\bm \beta}_{\text{ptl}} - \bm \beta\|^2$ can be decomposed into three parts. The first part $\sum_{k=1}^K s_{\max}\log  p/ n_k$ is from the estimation error of $\wh B$ and the second term $\{K+\log (p\wedge n)\}/n$ is from the error of $\wh {\bm w}_{\text{ptl}}$. Specifically, when $B$ is known, the OLS estimator has error $K/n$, and  the additional term $\log(p\wedge n)/n $ is introduced to guarantee the non-asymptotic bound holds with high probability.  The third part $(s_{\delta} r_n^2\wedge hr_n\wedge h^2)$ is the error of $\|\wh{\bm \delta}_{\text{ptl}} - \bm \delta\|^2$.  If $\sum_{k=1}^K s_k /n_k = o(K/n)$ and $K = O(\log p)$, then we have $r_n = O\{(\log p / n)^{1/2}\}$. Consequently, $\|\wh {\bm \delta}_{\text{ptl}} - \bm \delta\|^2$ has order $(s_{\delta}\log p/n)\wedge h(\log p/n)^{1/2}\wedge h^2$, which represents the optimal convergence rate of $\wh{\bm \delta}_{\text{ptl}}$ when $\bm w$ is known. 

Next we consider how the profiled transfer learning estimator improves the estimation performance. Recall that the convergence rate of the LASSO estimator in terms of $\ell_2$ norm computed on the target data only is $(s_0\log p/n)^{1/2}$ \citep{meinshausen2008lasso,bickel2009simultaneous,vershynin2018high,wainwright2019high}. Then Corollary \ref{Corollary compare lasso} provides some sufficient conditions under which the transfer learning estimator might converge faster than a standard LASSO estimator. 

\begin{corollary}
	Assume the same conditions in Theorem \ref{Theorem convergence rate}. Further assume that (1) $\sum_{k=1}^K s_{\max}/n_k = o(s_0/n)$, (2) $K = o(s_0\log p)$, and (3) (i) when $h > s_{\delta} r_n$,  $\bm \delta$ is sparse with $s_{\delta} = O(1)$, (ii) when $h < s_{\delta} r_n$, $h = o\{(s_0\log p/n)^{1/2}\}$, we have $\|\wh {\bm \beta}_{\textnormal{ptl}} - \bm \beta\| = o_p\{(s_0\log p/n)^{1/2}\}$. \label{Corollary compare lasso}
\end{corollary}
The detailed proof of Corollary \ref{Corollary compare lasso} is provided in Appendix D. The first condition requires that the estimation error of the source parameters is sufficiently small, which holds if the source parameters are sufficiently sparse or the source sample sizes are sufficiently large. The second condition constraints the number of source datasets to be much smaller than $s_0\log p$, which is easier to be satisfied for a relatively large $s_0$. The third condition requires that the residual part $\bm\delta$ should be sufficiently sparse with $s_{\delta} = O(1)$ or sufficiently small with $h = o\{(s_0\log p/n)^{1/2}\}$, which implies the \textit{approximate-linear} assumption holds. 

We compare our estimator with the Trans-LASSO estimator $\wh{\bm \beta}_{\text{tlasso}}$ proposed by \cite{li2020transfer}, which with high probability satisfies $$\|\wh{\bm \beta}_{\text{tlasso}}  - \bm\beta \|^2\lesssim  \frac{s_0\log p}{n + n_{\mathcal A}} + \Big(\frac{s_0\log p} n \Big)\wedge h_0 \Big(\frac{\log p}n\Big)^{1/2} \wedge h_0^2,$$
 where $\mathcal A = \{k: \|\bm \beta^{(k)} - \bm\beta\|_1 \le h_0 \}$ is the informative set, containing the sources whose parameters differ with target parameter by at most $h_0$ in $\ell_1$ norm, and $n_{\mathcal A} = \sum_{k\in \mathcal A} n_k$ represents the total sample sizes of the informative set. For simplicity, we only consider the case that $h_0 >h\vee (\log p / n)^{1/2}$ since the \textit{vanishing-difference} assumption is more stringent in applications. Corollary \ref{Corollary compare trans} provides some sufficient conditions under which the convergence rate of PTL estimator is faster than that of Trans-Lasso. 
 \begin{corollary}
 	 Assume the same conditions in Theorem \ref{Theorem convergence rate}. Further assume that $n_{\mathcal A} = \nu n^{\xi}$ for some constants $\nu >0 $ and $\xi>1$ and $s_0 n / n_{\mathcal A} = O(1)$. Define $m_0 = h_0n^{1/2}/s_0(\log p )^{1/2}$. Further assume that 
 	\begin{itemize}
 		\item [(1)] $\sum_{k=1}^K s_{\max}/n_k = o(s_0/n^\xi )$,
 		\item [(2)] $K = o(n^{-(\xi-1)}s_0\log p )$, and
 		\item [(3)]
 		\begin{itemize}
 			\item [(i)] If $h > s_{\delta}r_n$, it is required that 
 			$s_{\delta} = o\{s_0 ( m_0\wedge 1)\}$, 
 			\item [(ii)] If $h < s_{\delta }r_n$, it is required that 
 			$h = o\{(s_0\log p/n)^{1/2} (m_0^{1/2}\wedge 1)\}$,
 		\end{itemize} 
 	\end{itemize}
 	then we have $\|\wh {\bm \beta}_{\textnormal{ptl}} - \bm \beta\| = o_p(\|\wh{\bm \beta}_{\textnormal{tlasso}} - \bm \beta\|)$.  \label{Corollary compare trans}
 \end{corollary}

The detailed proof of Corollary \ref{Corollary compare trans} is provided in Appendix E. For simplicity, we consider the case $n \lesssim n_{\mathcal A}$. The first and second conditions hold when $n_{\mathcal A}$ is not too large. In this case, most of $\bm \beta^{(k)}$'s are quite different from $\bm \beta$. This is expected since when $h_0$ is small, it is hard for us to find enough sources with parameters satisfying the \textit{vanishing-difference} assumption.
 The third condition restricts the divergence rate of $s_\delta$ and $h$. When $\bm \delta$ is more sparse than $\bm\beta$ and $\|\bm\delta\|_1$ is quite small, the PTL estimator performs better than the Trans-LASSO estimator.

\csubsection{Minimax Lower Bound}

 Theorem \ref{Theorem convergence rate} provides a probabilistic upper bound for the convergence rate of the profiled transfer learning estimator $\wh{\bm \beta}_{\text{ptl}}$. Following \cite{li2020transfer} and \cite{tian2022transfer}, we are motivated to derive a minimax type lower bound, which  is independent of the specific estimating procedure used. This leads to the following theorem.

\begin{theorem}
	Assume the technical conditions (C2) and (C4). Further assume that (i) $\ve_i^{(k)}$'s and $\ve_i$'s are independently generated according to the normal distribution $N(0,\sigma^2_k)$ and $N(0,\sigma^2)$, respectively and  (ii) $\lambda_{\max}(\Sigma_k )\le C_{\max}$ for $k = 1,\dots, K$. Then, for some positive  constants $C_1$, $C_2$ and $C_3$, it holds  that \label{Theorem minimax}
	\beqrs
		\inf_{\wh{\bm \beta}} \sup_{\bm \beta \in \Theta(\bm s,h )} P\bigg(\|\wh{\bm \beta} -\bm \beta\|^2 &\ge &C_1\sum_{k=1}^K \frac{s_{\max}\log p}{n_k \vee n} + C_2 \frac{K}{n} \\&+& C_3 \Big(\frac{s_\delta\log p}{n}\Big)\wedge h\Big(\frac{\log p}{n}\Big)^{1/2}\wedge h^2\bigg) \ge \frac 12.
	\eeqrs
\end{theorem}
The detailed proof of Theorem \ref{Theorem minimax} is provided in Appendix G. The normality assumption of error term is used to facilitate the technical treatment \citep{raskutti2011minimax,rigollet2011exponential}. Combining Theorems \ref{Theorem convergence rate} and \ref{Theorem minimax}, we can draw the following conclusions. For the resulting $\wh {\bm \beta}_{\text{ptl}}$ to be minimax optimal, a sufficient condition is $r_n \le C_r (\log p/n )^{1/2}$ for some constant $C_r>0$, which leads to two additional requirements for the parameters. (i) The estimation errors due to the source parameter estimates must be well controlled in the sense that $\sum_{k=1}^K s_{\max}/n_k = O(1/n)$; 
(ii) the number of source datasets cannot be too large in the sense that $K = O(\log p)$. 
These conditions are expected; otherwise $\bm w$ and $\bm\delta$ cannot be estimated well.

Under the aforementioned two conditions (i) and (ii), we can rewrite the probabilistic upper bound of $\|\wh{\bm \beta}_{\text{ptl}} - \bm \beta\|^2$ and minimax lower bound for $\Theta(\bm s,h)$ as
\beqrs
\inf_{\theta \in \Theta(\bm s, h)}P\bigg(\|\wh{\bm \beta}_{\text{ptl}} - \bm \beta\|^2 &\lesssim& \sum_{k=1}^K \frac{s_{\max}\log p}{n_k} +  \frac{K+\log(p\wedge n)}n \\&+& \Big(\frac{s_\delta \log p}n\Big)\wedge h\Big(\frac{\log p}n\Big)^{1/2}\wedge h^2\bigg ) \ge 1 - \exp\{\log(p\wedge n)\}.
\eeqrs
It is then of great interest to study the conditions, under which the upper bound of $\|\wh{\bm \beta}_{\text{ptl}} - \bm \beta\|^2$ and minimax lower bound of our transfer learning problem match each other in terms of the convergence rate. This leads to the following different cases.

\textbf{Case 1.} The \textit{approximate-linear} assumption holds well in the sense $h< (\log p/n)^{1/2}$. 
This scenario implies that the target parameter can be approximated by the linear combination of source parameters with asymptotically negligible error. 
In this case, we have $(s_\delta \log p/n)\wedge h(\log p/n)^{1/2}\wedge h^2 = h^2 < \log p/n$. Next by Theorems \ref{Theorem convergence rate} and \ref{Theorem minimax}, we obtain the upper bound of order \[
\sum_{k=1}^K \frac{s_{\max}\log p}{n_k} + \frac{K+\log(p\wedge n)}n +  h^2\] and lower bound of order \[\sum_{k=1}^K \frac{s_{\max}\log p}{n_k \vee n} + \frac Kn + h^2.\] 
For simplicity, we assume that $n_k\ge n$, which is often the case in application. 
These two bounds match each other if one of the following condition holds: (i) $\log(p\wedge n)\lesssim K\lesssim \log p$ and (ii) $\log (p\wedge n) /n\lesssim(\sum_{k=1}^K s_{\max}\log p/n_k) \wedge h^2$ and $n \lesssim n_k$ for $k = 1\dots, K$.

\textbf{Case 2.} If $(\log p /n)^{1/2} \le h\le s_{\delta}(\log p/n)^{1/2}$, we then have $(s_\delta \log p/n)\wedge h(\log p/n)^{1/2}\\\wedge h^2 = h(\log p/n)^{1/2} \ge \log p/n$. Then the upper bound 
and the minimax lower bound match each other at the rate of $h(\log p/n)^{1/2}$. 
This allows the $\bm \delta$-vector to be relatively dense in the sense that $s_{\delta} \ge h(\log p/n)^{-1/2}$. 

\textbf{Case 3.} If $h > s_{\delta}(\log p/n)^{1/2}$, we then have $(s_\delta \log p/n)\wedge h(\log p/n)^{1/2}\wedge h^2 = s_\delta \log p/n \ge \log p/n$. Accordingly, the upper bound 
and the minimax lower bound 
match each other at the rate of $s_\delta \log p/n$. In this case, the $\bm \delta$-vector is sparse in the sense $\|\bm \delta\|_0 < h(\log p/n)^{-1/2}$. 

To summarize, the minimax lower bound of the transfer learning problem of interest can be achieved by our proposed PTL estimator $\wh{\bm \beta}_{\text{ptl}}$, provided that $\bm \delta$ is either sufficiently sparse or reasonably small but detectable in $\ell_1$-norm.

 \csection{NUMERICAL RESULTS}

 \csubsection{Preliminary Setup}

 To evaluate the finite sample performance of $\wh{\bm \beta}_{\text{ptl}}$, we present here a number of simulation experiments. Specifically, we examine the finite sample performance of three estimators:    (1) the LASSO estimator using the target data only, denoted as  $\wh{\bm \beta}_{\text{lasso}}$, (2) the Trans-LASSO estimator proposed in \cite{li2020transfer} $\wh{\bm \beta}_{\text{tlasso}}$, and (3) our PTL estimator $\wh{\bm \beta}_{\text{ptl}}$.

 In the whole study, we fix $K = 5$, $p = 500$ and $h = 6$. The source data are i.i.d. observations generated from the model $Y_i^{(k)} = X_i^\top {\bm \beta}^{(k)} + \ve_i^{(k)}$, where $X_i^{(k)} \sim N(0,\Sigma_k)$ and $\ve_i^{(k)} \sim N(0,\sigma_k^2)$. The target data are generated similarly but with $\cov(X_i) = \Sigma$, $\var(\ve_i) = \sigma^2$ and the regression coefficient to be $\bm \beta$. The parameters $\bm \beta^{(k)}$, $\bm\beta$, $\sigma_k^2$, $\sigma^2$, $\Sigma_k$ and $\Sigma$ will be specified later for different scenarios. For each setting, we replicate the experiment for a total of $B = 500$ times.

For the PTL estimator, five-fold cross validation is used again to choose all the shrinkage parameters $\lambda_k$'s and $\lambda_\delta$.  We use a generic notation $\wh{\bm \beta}^{[b]} = (\wh\beta^{[b]}_1,\dots,\wh\beta^{[b]}_p)^\top  \in \mR^p$ to represent one particular estimator (e.g. the PTL estimator) obtained in the $b$th replication with $1\le b\le B$. We next define the mean squared error (MSE) as $\mbox{MSE}(\wh{\bm \beta}^{[b]})= p^{-1}\|\wh{\bm \beta}^{[b]} -\bm  \beta\|^2 = p^{-1}\sum_{j=1}^p (\wh\beta^{[b]}_j - \beta_j)^2$. This leads to a total of $B$ MSE values for each simulation example, which are then summarized by boxplots.

\csubsection{Simulation Models}

We consider here four simulation examples. In Example 1, we set $\supp(\bm\beta^{(k)})$'s to be nearly the same as $\supp(\bm \beta)$ \citep{li2020transfer}. In Example 2, $\supp(\bm\beta^{(k)})$'s are significantly different from the $\supp(\bm\beta)$. For the first two examples, $\Sigma_k$'s are identical to $\Sigma$ and $\sigma_k$'s to $\sigma$. In Example 3, $\Sigma_k$'s are significantly different from $\Sigma$, and $\sigma_k$'s from $\sigma$. In example 4, we evaluate the case when data generation procedure violates the identification condition. 

\begin{figure}[t]
	\centering
	\subfloat[$N=1500$ and $s=40$ ]{\includegraphics[width=0.5\textwidth]{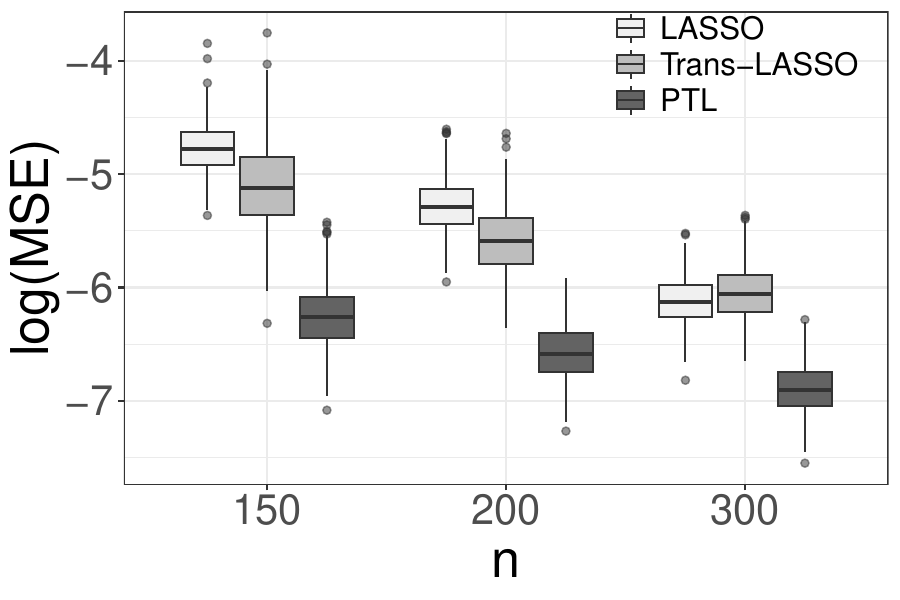}}
	\subfloat[$n=200$ and $s = 40$ ]{\includegraphics[width=0.5\textwidth]{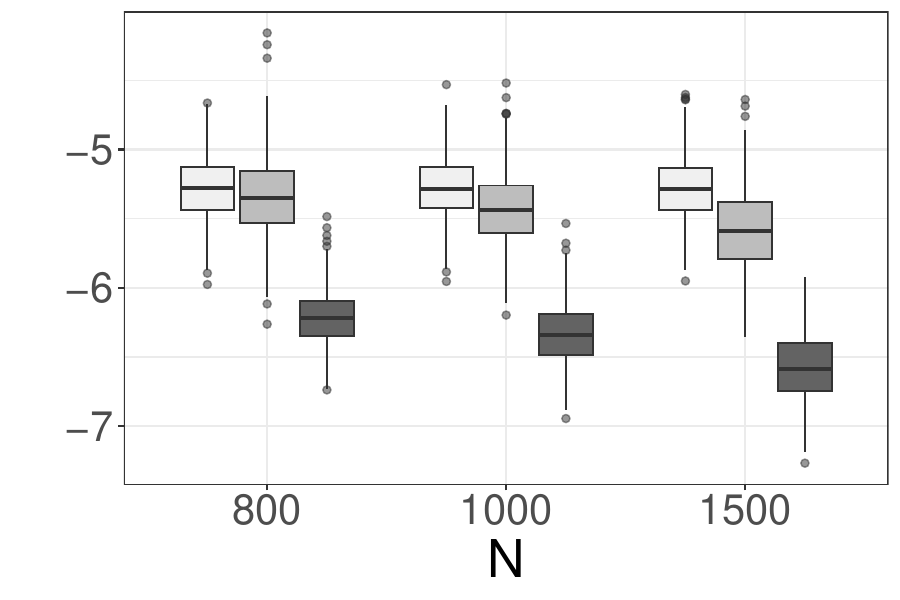}}\\
	\subfloat[$N = 1500$ and $s = 300$ ]{\includegraphics[width=0.5\textwidth]{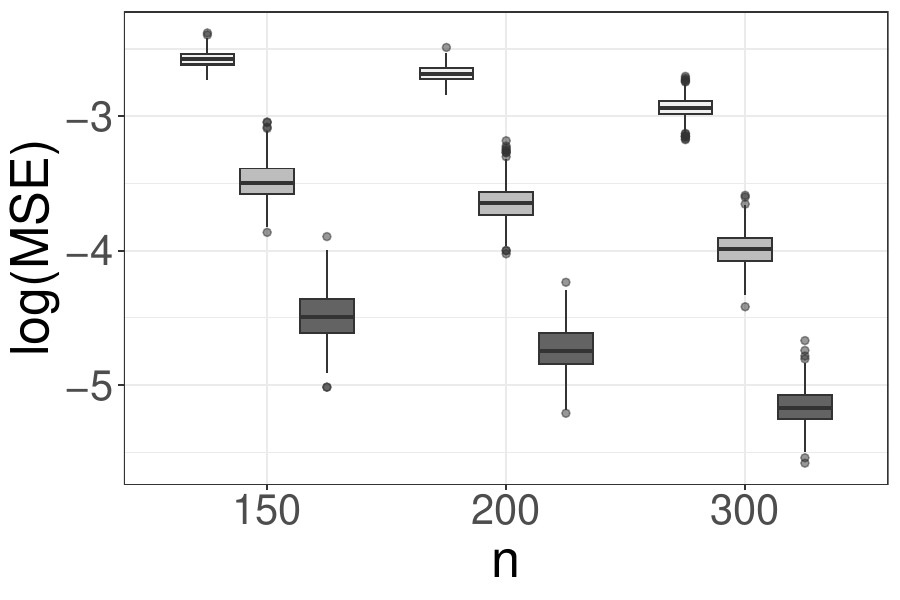}}
	\subfloat[$n=200$ and $s = 300$  ]{\includegraphics[width=0.5\textwidth]{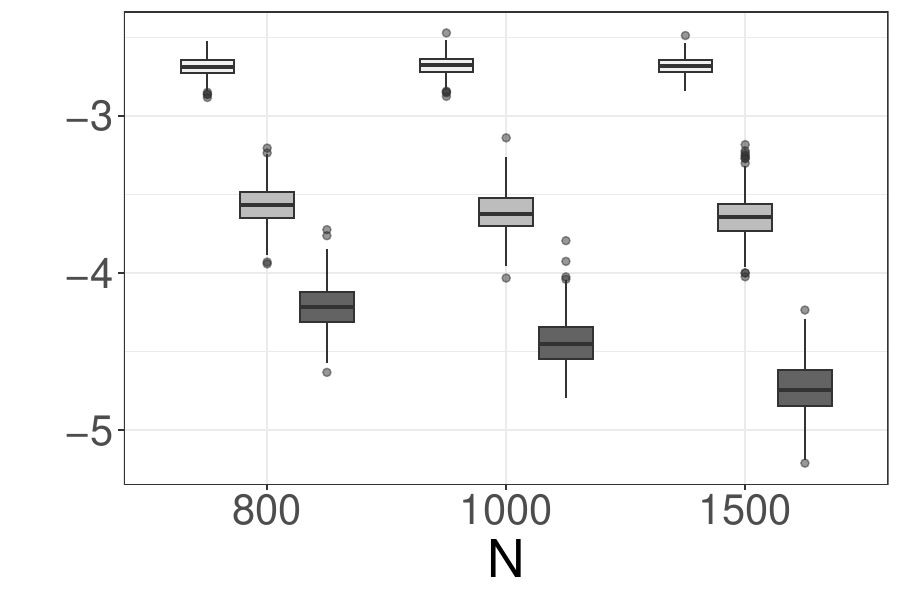}}
	\caption{The boxplots of the log-transformed MSE values for three different estimators $\wh\beta_{\text{lasso}}$,  $\wh\beta_{\text{tlasso}}$ and $\wh\beta_{\text{ptl}}$ in Example 1. The first row presents the sparse cases with $s = 40$ and the second row presents the dense cases when $ s = 300$. The first and second columns study the target sample size $n$ and the source sample size $N$, respectively.}
	\label{Figure source common support}
\end{figure}

\textbf{Example 1.} This is an example revised from \cite{li2020transfer} and \cite{tripuraneni2021provable}. Specifically, we generate $\bm\beta$ by the following two steps.
\begin{itemize}
	\item Generate $B$. Let $\Omega = (\omega_{ij})_{r_0\times K}$ be an $r_0\times K$ matrix with $\omega_{ij}$ independently generated from standard normal distribution. Let $U_K= (u_1,\dots,u_K)\in\mR^{r_0\times K}$ be the first $K$ left singular vector of $\Omega$ with $\|u_k\|=1$ for $1\le k\le K$. Then set $B = (2U_K^\top,  0.3\bm1_{s-r_0,K}^\top , \bm 0_{p-s,K}^\top)^\top \in \mR^{p\times K}$ \citep{tripuraneni2021provable}.
	\item Generate $\bm\delta$ and $\bm\beta$. Specifically, $\bm\delta = (\delta_1,\dots,\delta_p)^\top $ is generated as follows. Let $\mS$ be an index set with $|\mathcal{S}| = s_{\delta}$ randomly sampled from $\{s+1,\dots,p\}$ without replacement. Next, for each $j\in \mS$, we generate $\delta_j$ independently from $N(0,h/s_{\delta})$. For each $j\notin \mS$, set $\delta_j = 0$ \citep{li2020transfer}. The target parameter is then assembled by $\bm \beta = B\bm w + \bm \delta$ with $\bm w = (3/2,3/4,0,0,-5/4)^\top\in  \mR^5$. When $s_\delta$ is small, the difference between $\supp(\bm \beta)$ and $\supp(\bm\beta^{(k)})$ is only $s_{\delta}$ elements.
\end{itemize}

Set 
$\Sigma_k = \Sigma = I_p$ \citep{tripuraneni2021provable} and $\sigma_k^2 = \sigma^2 = 1$ for $1\le k\le K$. Fix the source sample size $n_k \equiv N$ to be the same for $1\le k\le K$. Consider different $n\in \{150,200,300\}$ and $N\in \{800,1000,1500\}$. 
Set $r_0 = \lfloor s/3\rfloor$ and $s_{\delta} = \lfloor s/5\rfloor$ with $s=40$ (sparse) and $s=300$ (dense), respectively. Here $\lfloor x\rfloor$ stands for the largest integer no greater than $x$. The experiment is then randomly replicated for a total of $B = 500$ times. This leads to a total of 500 MSE values, which are then log-transformed and boxplotted in Figure \ref{Figure source common support}. 

By Figure \ref{Figure source common support}, we find that the MSE values of both the Trans-LASSO and the PTL estimators are consistently smaller than those of the LASSO estimator. This is reasonable because LASSO uses only the target data, while Trans-LASSO and PTL utilize the information from both the target and the sources.   
Secondly, 
 MSE values of Trans-LASSO and PTL decrease as $N$ increases; see the right top and right bottom panels. This suggests that a larger source sample size also helps to improve the estimation accuracy of the target parameter significantly for the Trans-LASSO and PTL method. Lastly and also most importantly, for all the simulation cases, the MSE values of the PTL estimator are clearly smaller than those of the other estimators. This suggests that the PTL estimator is the best choice compared with two other competitors for this particular case.

\textbf{Example 2.} For this example, we allow the supports of the source and target parameters to be very different. Specifically, we generate $\bm\beta^{(k)}$'s and $\bm\beta$ by the following two steps.

\begin{itemize}
	\item Generate $B$. Specifically, we generate an $s\times K$ random matrix $\Omega$ in the same way as in Example 1. Denote its first $K$ left singular vectors by $U_K = (u_1,\dots, u_K) \in \mR^{s\times K}$ with $u_k = (u_{k1},\dots, u_{ks})^\top \in \mR^s$ and $\|u_k\| = 1$. The support of $\bm \beta^{(k)}$ is given by $\mS_c \cup \mS_k$, where $\mS_c$ denotes the common support set shared by all sources and $\mS_k$ stands for the support set uniquely owned by source $k$. Define $s_c=|\mS_c |$ the cardinality of $\mS_c$. Let $\mS_k = \{s_{k,1},\dots, s_{k,s-s_c}\}$ be a random subset of $\{s_c+1, \dots, p\}$ without replacement. Set $\beta_j^{(k)} = 2u_{kj}$ for $1\le j\le s_c$,  $\beta_{s_{k,j}}^{(k)} = 2u_{k,j+s_c}$ for $1\le j \le s - s_c$ and $\beta^{(k)}_j = 0$ if $j\notin \mS_c \cup  \mS_k$. Set $B = (\bm \beta^{(1)},\dots, \bm \beta^{(K)}) \in \mR^{p\times K}$.
	\item Generate $\bm\delta$ and $\bm\beta$. Let $\mS_0$ be a random subset of  $\{1,\dots, p\}$ without replacement with $|\mS_0| = s_{\delta}$. Let $\wt \delta \in \mR^{s_\delta}$ be the last eigenvector of $B_{\mS_0} B_{\mS_0}^\top$, where $B_{\mS_0} = (\beta^{(1)}_{\mS_0}, \dots, \beta^{(K)}_{\mS_0} )$ and $\beta^{(k)}_{\mS_0} = (\beta_j:j\in\mS_0)^\top \in\mR^{s_{\delta}}$. Fix $\bm \delta$ to be $\delta_{\mS_0} = h \wt\delta/ |\wt\delta|$ and $\delta_{\mS_0^c} = 0$. Therefore, the identification condition about delta (i.e. $B^\top \Sigma \bm \delta = 0$) can be naturally satisfied. The target parameter is then assembled as $\bm \beta = B\bm w +\bm \delta$ with $w = (3/2,3/4,0,0,-5/4)^\top\in  \mR^5$.
\end{itemize}

The specification of $\Sigma,\Sigma_k, \sigma^2$ and $\sigma_k^2$ are the same as in Example 1. The experiment is then randomly replicated for a total of $B = 500$ times. This leads to a total of 500 MSE values, which are then log-transformed and boxplotted in Figure \ref{Figure source different support}. 

\begin{figure}[t]
	\centering
	\subfloat[$n_s=800$ and $s=40$ ]{\includegraphics[width=0.5\textwidth]{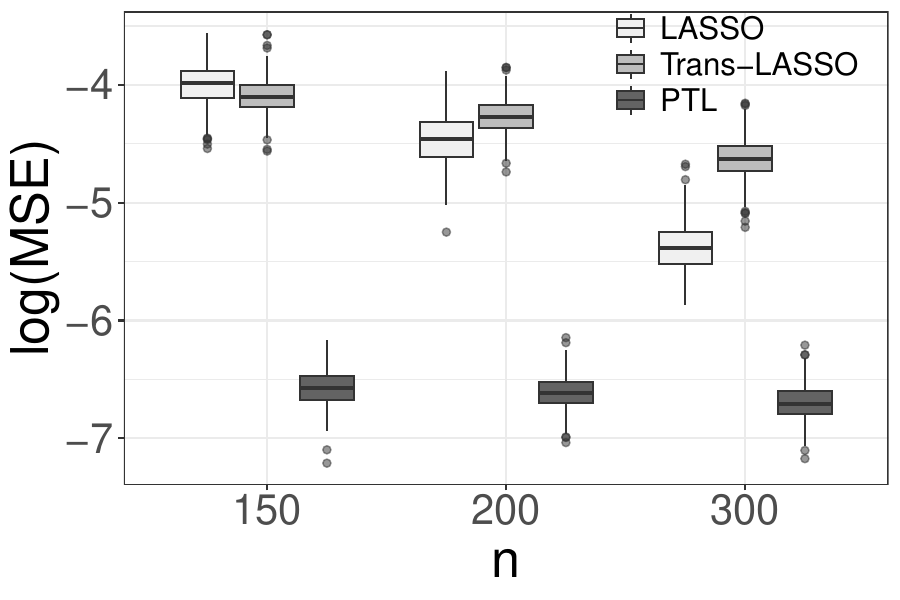}}
	\subfloat[$n=200$ and $s=40$ ]{\includegraphics[width=0.5\textwidth]{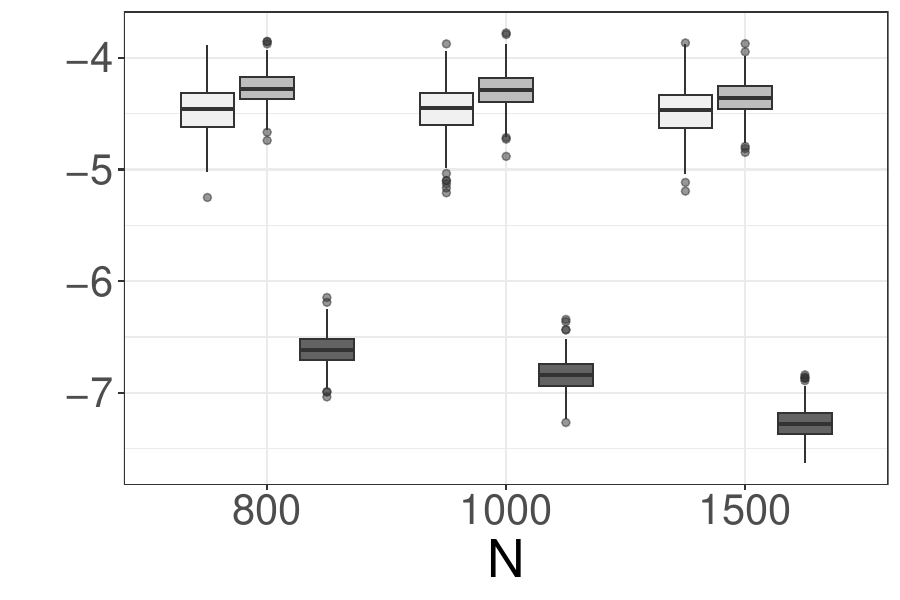}}\\
	\subfloat[$n_s=800$ and $s=300$ ]{\includegraphics[width=0.5\textwidth]{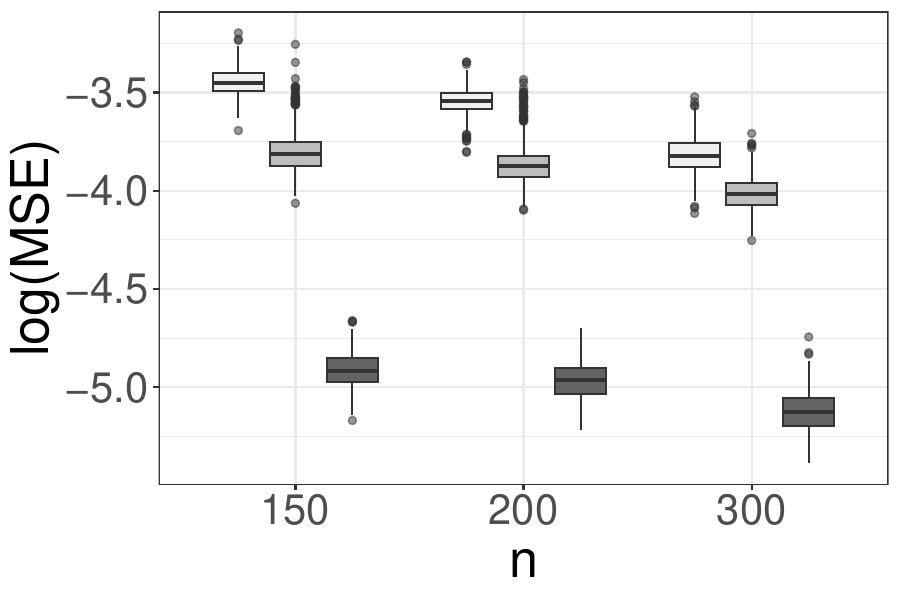}}
	\subfloat[$n=200$ and $s=300$ ]{\includegraphics[width=0.5\textwidth]{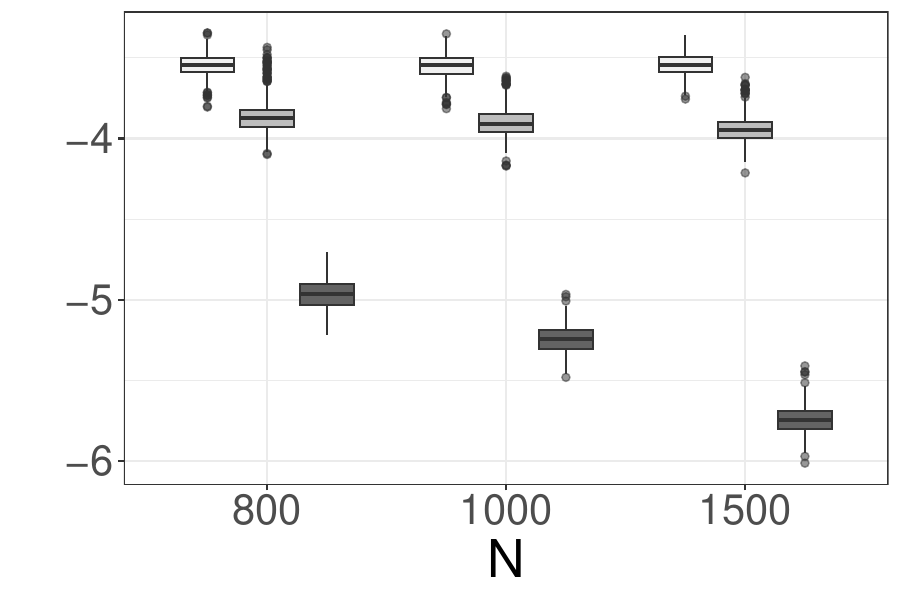}}
	\caption{The boxplots of the log-transformed MSE values for four different estimators $\wh\beta_{\text{lasso}}$, $\wh\beta_{\text{tlasso}}$ and $\wh\beta_{\text{ptl}}$ in Example 2. The first row presents the sparse cases with $s = 40$ and the second row presents the dense cases when $ s = 300$. The first and second columns study the target sample size $n$ and the source sample size $N$, respectively.}
	\label{Figure source different support}
\end{figure}

Most numerical findings obtained by Figure \ref{Figure source different support} are qualitatively similar to those of Figure \ref{Figure source common support}. The key differences are given by the right two panels, where as the source data sample size $N$ increases, the decreases of $\log(\mbox {MSE})$ for the Trans-LASSO estimator is small. This suggests that the transfer learning capability of the Trans-LASSO estimator is very limited for this particular example. 
The success of the Trans-LASSO method heavily hinges on the  \textit{vanishing-difference} assumption, which is unfortunately violated seriously in this example. Similar patterns are also observed for the LASSO estimator. 
In contrast, the $\log(\mbox {MSE})$ values for PTL estimator steadily decrease as $n_s$ increases. The PTL method once again stands out clearly as the best method for this particular case.

\begin{figure}[t]
	\centering
	\subfloat[$N=1500$ and $s=30$ ]{\includegraphics[width=0.5\textwidth]{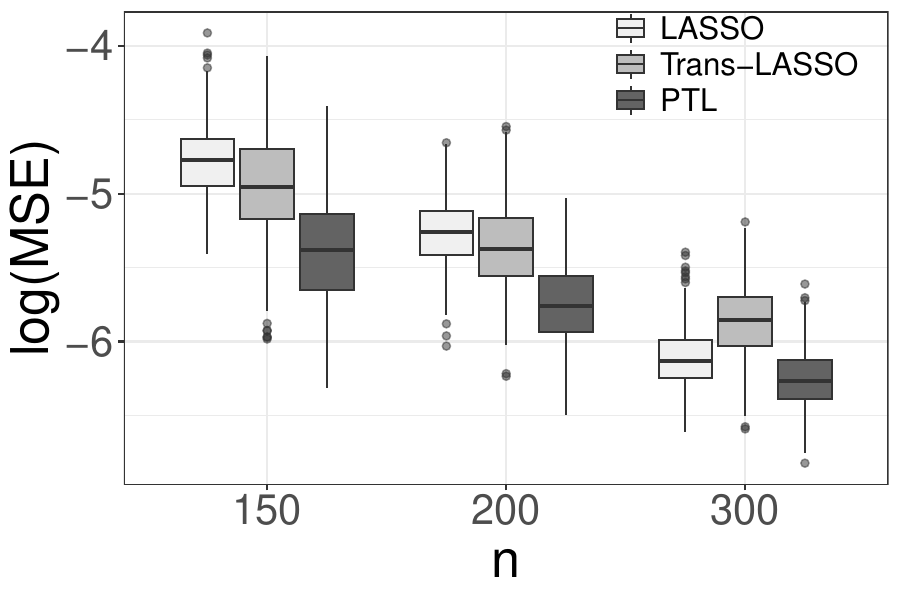}}
	\subfloat[$n=200$ and $s=30$ ]{\includegraphics[width=0.5\textwidth]{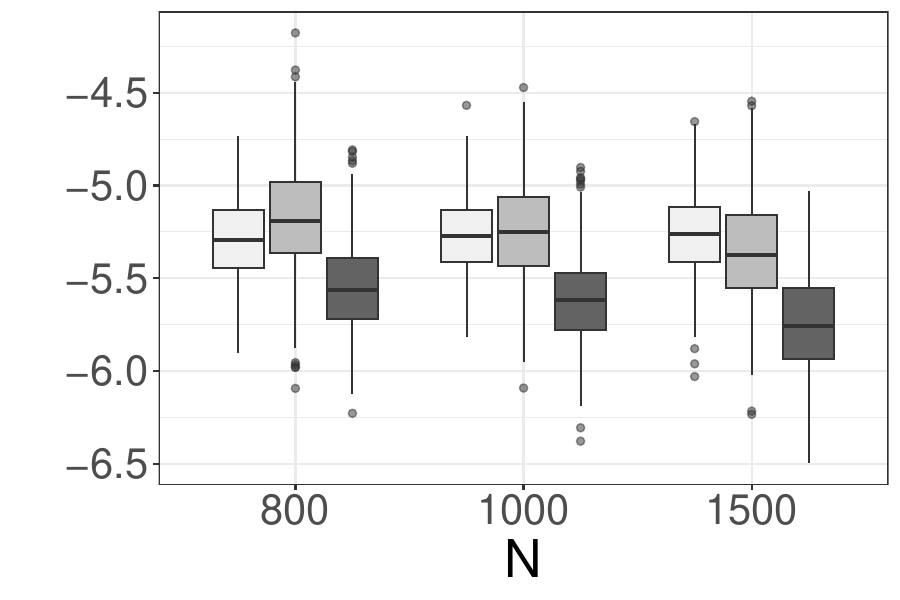}}\\
	\subfloat[$N = 1500$ and $s=300$ ]{\includegraphics[width=0.5\textwidth]{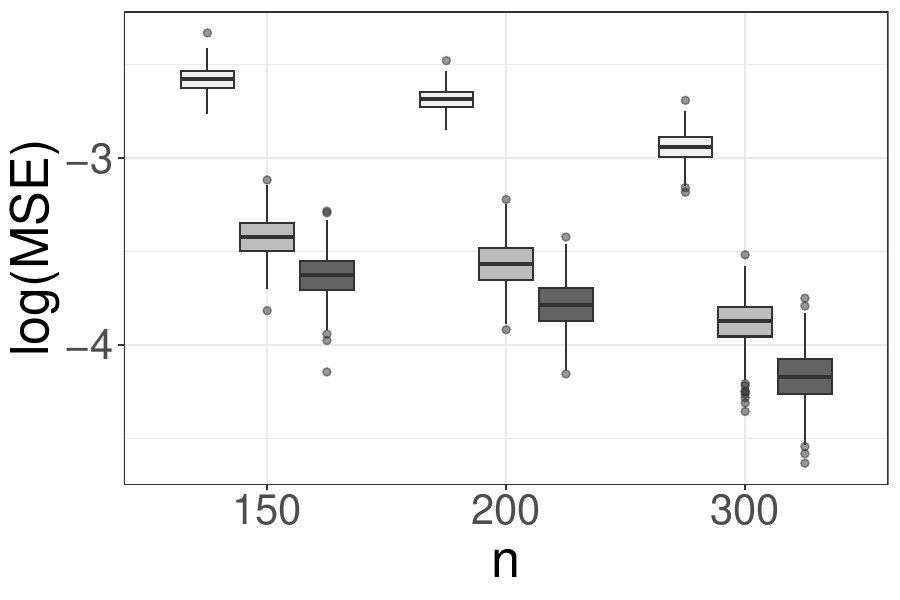}}
	\subfloat[$n=200$ and $s=300$ ]{\includegraphics[width=0.5\textwidth]{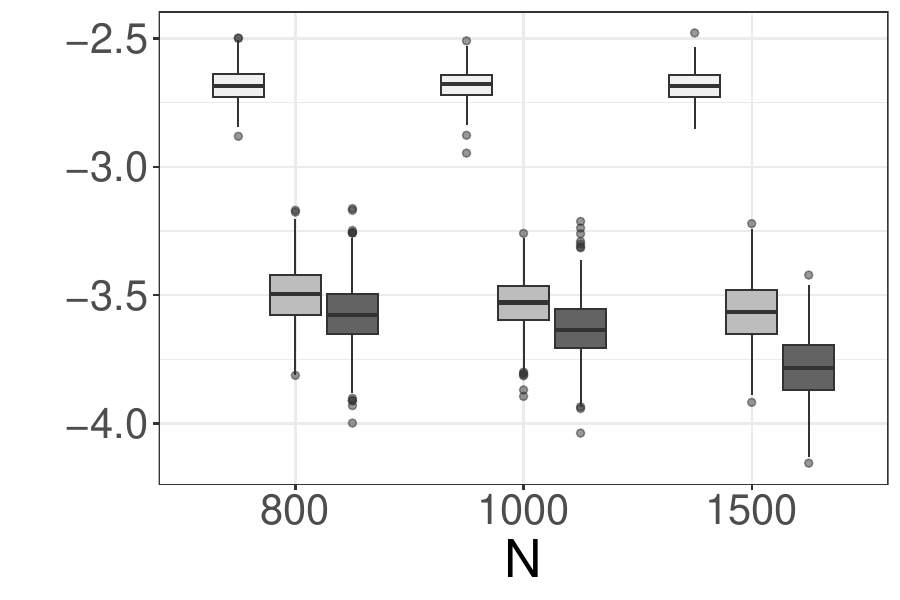}}
	\caption{The boxplots of the log-transformed MSE values for four different estimators $\wh\beta_{\text{lasso}}$, $\wh\beta_{\text{tlasso}}$ and $\wh\beta_{\text{ptl}}$ in Example 3. The first row presents the sparse cases with $s = 40$ and the second row presents the dense cases when $ s = 300$. The first and second columns study the target sample size $n$ and the source sample size $N$, respectively.}
	\label{Figure source common support heterogenous covariance}
\end{figure}

\textbf{Example 3.} For this example, the parameter specifications here are the same as those of Example 1. The key difference is that the covariance matrices of covariates are allowed to be different for the target and source datasets. Moreover, the error variances are also allowed to be different for different datasets. Specifically, we generate $\Sigma_k$, $\Sigma$, $\sigma_k^2$ and $\sigma^2$ as follows.

\begin{itemize}
	\item Generate $\Sigma_k$ and $\Sigma$. For source $k$, set $\Sigma_k$  to be a symmetric Toeplitz matrix with first row given by $(1,\bm 1_{2k-1}^\top / (k+1), \bm 0_{p-2k}^\top )$, where $\bm 1_{2k-1} = (1,\dots,1)^\top \in \mR^{2k-1}$ for $1\le k\le K$; see Section 8.3.4 of \cite{seber2008matrix} for more details. For the target dataset, set the covariance matrix of the covariates as $\Sigma=I_p$ \citep{li2020transfer}.
	\item Generate $\sigma_k^2$ and $\sigma^2$. Set the error variance of the $k$th source as $\sigma_k^2 = k$. Fix $\sigma^2 = 1$ for the target data.
\end{itemize}
The experiment is then randomly replicated for a total of $B = 500$ times. This leads to a total of 500 MSE values, which are then log-transformed and boxplotted in Figure \ref{Figure source common support heterogenous covariance}. The numerical findings obtained in Figure \ref{Figure source common support heterogenous covariance} are similar to those of Figure \ref{Figure source common support}. The proposed PTL estimator remains to be the best one.

\begin{figure}[t]
	\centering
	\subfloat[$N=1500$ and $s=30$ ]{\includegraphics[width=0.5\textwidth]{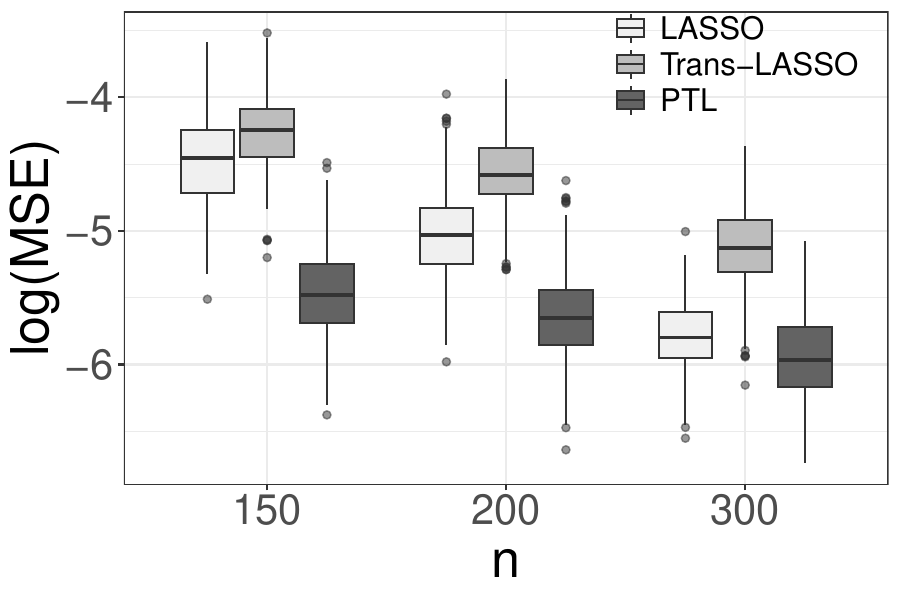}}
	\subfloat[$n=200$ and $s=30$ ]{\includegraphics[width=0.5\textwidth]{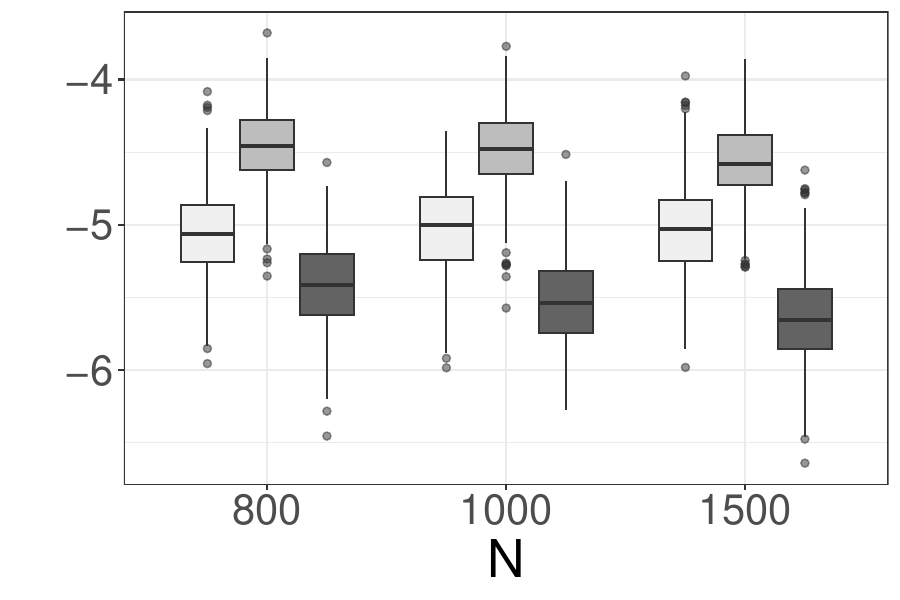}}\\
	\subfloat[$N = 1500$ and $s=300$ ]{\includegraphics[width=0.5\textwidth]{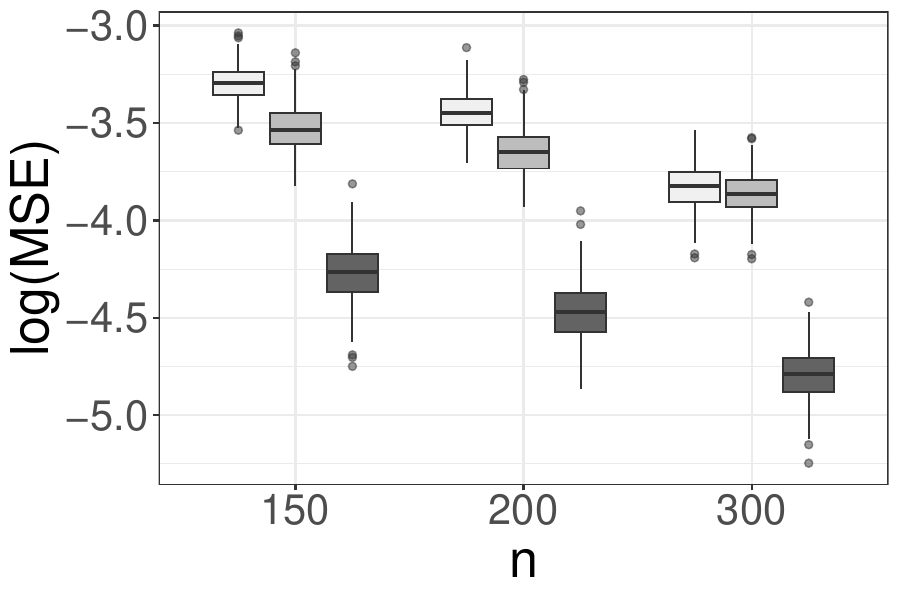}}
	\subfloat[$n=200$ and $s=300$ ]{\includegraphics[width=0.5\textwidth]{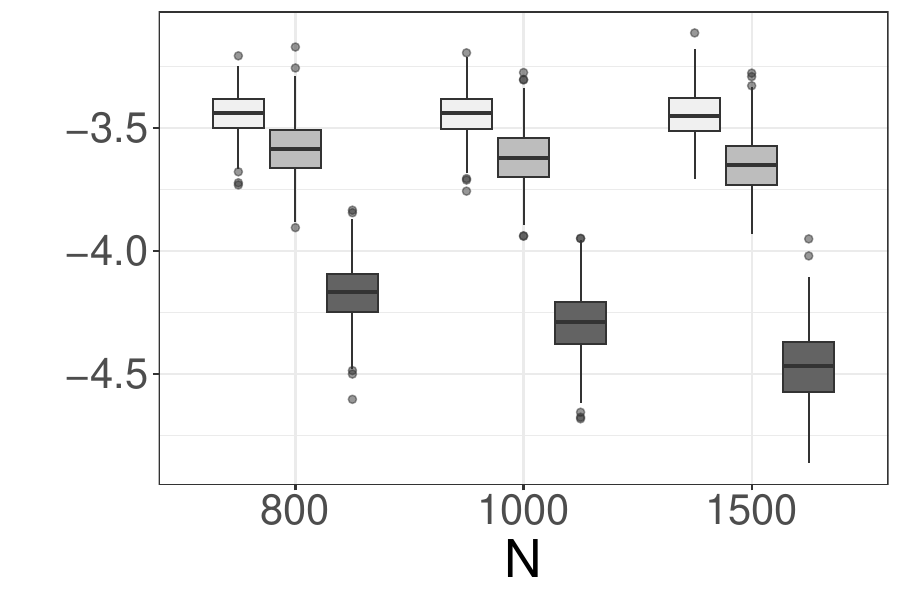}}
	\caption{The boxplots of the log-transformed MSE values for four different estimators $\wh\beta_{\text{lasso}}$, $\wh\beta_{\text{tlasso}}$ and $\wh\beta_{\text{ptl}}$ in Example 4. The first row presents the sparse cases with $s = 40$ and the second row presents the dense cases when $ s = 300$. The first and second columns study the target sample size $n$ and the source sample size $N$, respectively.}
	\label{Figure delta not sparse}
\end{figure}

\textbf{Example 4} In this example, we consider the case the parameter generation is different from the identification condition. Specifically, the parameter $\bm \beta^{(k)}$'s and $\bm\beta$ are generated as follows.

\begin{itemize}
	\item Generate $B$. The way for generating $B$ is the same as Example 1. Generate $\bm\delta$ and $\bm\beta$ as follows. The generation of $\bm\delta^*$ is similar to that of $\bm\delta$ in Example 1 expect for $\mS$ is sampled from $\{1,\dots,p\}$. Set $\bm\beta = B\bm w^* + \bm\delta^*$ with $\bm w^* = (3/2,3/4,0,0,-5/4)^\top\in  \mR^5$.
	\item Generate $\Sigma_k$, $\Sigma$, $\sigma_k^2$ and $\sigma^2$. Set $\Sigma_k$'s the same as those in Example 3. Set $\Sigma = (\sigma_{ij})_{p\times p}$ with $\sigma_{ij} = 0.5^{|i-j|}$. In this scenario, the identification condition $B^\top \Sigma \bm \delta^*=\mathbf{0}$ does not hold.  Set $\sigma_k^2 = 1.25^{k-3}$ and $\sigma^2 = 1$.
\end{itemize}
The experiment is then randomly replicated for a total of $B = 500$ times. This leads to a total of 500 MSE values, which are then log-transformed and boxplotted in Figure \ref{Figure delta not sparse}. We find that although the identification condition is violated, the proposed PTL estimator remains to be the best estimator.

\csubsection{Real Data Analysis}

We study here the problem of sentence prediction for fraud crimes occurred in China rendered in the year from 2018 to 2019. Fraud is one of the most traditional types of property crime, characterized by a high proportion and a wide variety. For example, from 2018 to 2019, fraud accounted for approximately 1\% of all recorded criminal activities. In Chinese criminal law, besides the basic fraud crime, offenses related to fraud crime include  contract fraud, insurance fraud, credit card fraud, and many other types. As the societal harm caused by different categories of fraudulent behavior varies, the sentencing guidelines also differ. Consequently, it becomes particularly important to make scientifically reasonable sentencing judgments for a given case of fraud, serving as a crucial guarantee for judicial fairness and justice. However, given the complexity and variety of different fraud cases, judges need to make their final judicial rulings based on the facts of the crime, relevant legal statutes, and their necessary personal judgment. The involvement of personal subjective judgment often leads to considerable uncertainty in judicial decisions, indirectly impacting the fairness and justice of judicial practices. Therefore, it is crucial to construct an objective, precise, and automated sentencing model to assist judges in making more scientifically stable and fair sentencing decisions.

\begin{figure}[t]
	\centering
	\includegraphics[width = 0.95\textwidth]{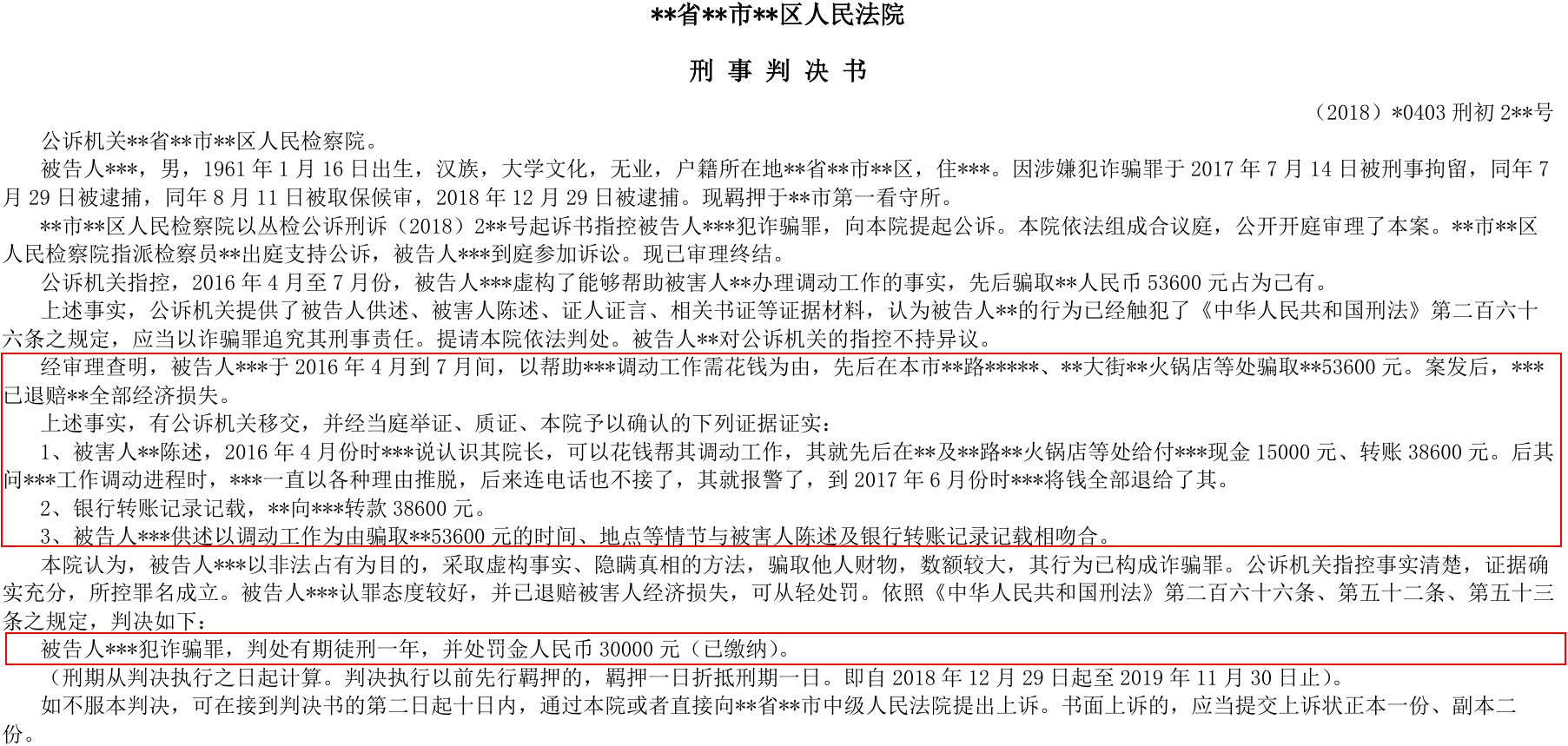}
	\caption{An illustration for judgment document. The middle box highlights the major facts. The bottom box highlights the sentencing decisions. }
	\label{Figure document example}
\end{figure}

To this end, we consider training different linear regression models for different types of fraud. Specifically, our data are collected from China Judgment Online (CJO, \url{https://wenshu.court.gov.cn/}) during the period from 2018 to 2019. For simplicity, only the first trials and fixed-term imprisonment cases are considered. Here, we use the log-transformed length of sentence as the dependent variable, which is extracted from the document; see the bottom box in Figure \ref{Figure document example} for illustration. Subsequently, we convert the text in the sentencing documents related to the criminal facts (see middle box in Figure \ref{Figure document example}) into high-dimensional 1536-dimensional vectors using the OpenAI text-embedding-ada-02 model (\url{https://openai.com/blog/new-and-improved-embedding-model}). After that, different high-dimensional linear regression models can be established for various types of fraud. It is worth noting that the proportions of different types of fraud in the entire sample are very different from each other. A few of the most common types of fraud occupy a large proportion of the whole sample. For instance, the crime of fraud accounts for 71.8\% of the total fraud cases and crime of contract fraud accounts for 9.7\%; see Table \ref{Table crimes} for details. Therefore, the sample sizes for these few types of frauds are relatively sufficient for accurate estimation. Unfortunately, there are a large number of other types of frauds with tiny sample sizes. Consider for example the crime of negotiable instrument fraud, the corresponding sample size is only 112. This tiny sample size is obviously insufficient to support precise parameter estimation. 

\begin{table}[!t]
	\centering
	\begin{threeparttable}[b]
			\caption{Different Types of Fraud Crimes}
			\begin{tabular}{ccc}
			\hline
			\hline
			Crime Type & \quad Sample Size $n$ \quad & \quad Percentage (\%)\quad  \\
			\hline
			Crime of Fraud & 37,121 & 71.80 \\
			Crime of Contract Fraud & 5,014 & 9.70 \\
			\quad\quad  Crime of Making out False Invoices\tnote{1}\quad \quad & 3,827 & 7.40 \\
			Crime of Credit Card Fraud & 3,816 & 7.38\\
			
			Crime of Cheating and Bluffing & 586 & 1.13\\
			Crime of Fund-Raising Fraud & 537& 1.04\\ 
			
			Crime of Insurance Fraud & 373 &0.72\\
			Crime of Loan Fraud & 192 & 0.37\\ 
	Crime of Defrauding Loans\tnote{2}& 126 & 0.24\\
			\multicolumn{1}{c}{\begin{tabular}[c]{@{}c@{}}Crime of Negotiable Instrument Fraud\end{tabular}} & 112&0.22 \\
			
			\hline 
		\end{tabular}
		\begin{tablenotes}
			\item[1] The full name is Crime of Making out False Value-Added Tax Invoices or Other False Invoices to Obtain an Export Tax Refund or to Offset Payable Taxes.
			\item[2] The full name is Crime of Defrauding Loans, Acceptance of Negotiable Instruments  and Financial Document or Instrument. 
		\end{tablenotes}
			\label{Table crimes}
	\end{threeparttable}

\end{table}

To overcome this challenge, we consider here the method of transfer learning. Specifically, we treat fraud categories with sample sizes exceeding 1000 as the source datasets, while the other small categories as target datasets. This leads to a total of 4 source datasets, corresponding to the crime of fraud, crime of contract fraud, crime of making out false invoices, and crime of credit card fraud. The corresponding sample sizes are given by 37,121, 5,014, 3,827 and 3,816, respectively. Then, we attempt to employ the PTL estimator proposed in this paper to borrow the information from the source datasets to help the target models. 
To evaluate the prediction performance of the proposed PTL method, we apply the two-fold cross validation method. For the training dataset, we estimate the regression coefficients by the PTL method. For the validation dataset, we evaluate the prediction performance by out-of-sample $R$-squared, defined as $R^2_{\text{val}} = 1 - \sum_{i\in \mS_{\text{val}}}(\wh Y_i - Y_i)^2 / \sum_{i\in \mS_{\text{val}}}(Y_i - \bar Y_{\text{val}})^2$. Here $\mS_{\text{val}}$ is the index of validation dataset and $ \bar Y_{\text{val}} $ is the sample mean of the validation response. This procedure is replicated for a total of 100 times, resulting in a total of 200 out-of-sample $R$-squared values. For comparison purpose, the LASSO and Trans-LASSO estimators are also tested. The average $R$-squared values are reported in Table \ref{Table crime R square}. By Table \ref{Table crime R square}, we find that the PTL estimators outperform the other two estimators for all target crimes. 

\begin{table}[!t]
	\centering
	\begin{threeparttable}[b]
			\caption{Average out-sample $R$-squared (\%) for different methods. }
		\begin{tabular}{ccccc}
			\hline
			\hline
			Crime Type&$n$ & LASSO & Trans-LASSO &  PTL \\
			\hline
			Crime of Cheating and Bluffing& 586&12.41& 16.85&25.90 \\ 
			Crime of Fund-Raising Fraud & 537& 2.16 & 4.83 & 12.90\\
			Crime of Insurance Fraud & 373& 3.78 &  12.05 & 14.64 \\
			Crime of Loan Fraud & 192&-0.34&  10.95&18.60\\ 
Crime of Defrauding Loans\tnote{1}& 126 &-5.19& 0.56& 1.78 \\
		Crime of Negotiable Instrument Fraud & 112& 3.94&27.92&30.08\\
			\hline 
		\end{tabular}
		\begin{tablenotes}
			\item[1] The full name of the crime is Crime of Defrauding Loans, Acceptance of Negotiable Instruments  and Financial Document or Instrument. 
		\end{tablenotes}
		\label{Table crime R square}
	\end{threeparttable}
\end{table}

\csection{CONCLUDING REMARKS}

In this article, we present here a profiled transfer learning method. The non-asymptotic upper bound and minimax lower bound are established. The proposed estimator is minimax optimal under appropriate conditions. Numerical studies and a real data example about sentence prediction are presented to demonstrate the finite sample performance of the proposed method. To conclude this article, we consider several interesting topics for future research. First, we use all the source datasets to construct the profiled response. However, the source datasets could be useless if its regression pattern is extremely different from the target one. Therefore, there is a practical need to select the useful source datasets by appropriate method. This seems the first interesting topic for future study. Second, we studied here linear regression model only. Then how to generalize the method of profiled transfer learning to other statistical models (e.g. a logistic regression model) is another interesting topic for future study. Lastly, high dimensional covariance estimation is another important statistical problem suffering from the curse of dimensionality. Then how to conduct profiled transfer learning for high dimensional covariance estimator is also an important topic worthwhile pursuing.  

\clearpage

\renewcommand\refname{\begin{center}
		\large{REFERENCES}
\end{center}}
\bibliographystyle{asa}
\bibliography{reference}
\newpage
	\begin{center}
	{\bf\Large Supplementary Materials to
		``Profiled Transfer Learning for High Dimensional Linear Model''}\\
	\bigskip

\end{center}


\appendix
\renewcommand{\csection}[1]
{\begin{center}
		\stepcounter{section}
		{\bf\large Appendix \Alph{section}. #1}
	\end{center}
}

\csection{Proof of Proposition \ref{Prop identification}}

Recall that $\bm\beta = B\bm w +\bm\delta$, we can rewrite the model as $Y_i = X_i^\top B \bm w + X_i^\top \bm\delta + \ve_i$. Multiply $B^\top X_i$ on both sides and take expectation, we have $B^\top E(X_i Y_i) = B^\top \Sigma B \bm w$. Note that $B^\top \Sigma B$ is nonsingular since the $\bm \beta^{(k)}$'s are linearly independent and $\Sigma$ is positive definite. As a result, the vector $\bm w$ be can identified as $\bm w = (B^\top \Sigma B)^{-1} B^\top E(X_i Y_i) $.

Next we consider the identification of $\bm \delta$. Multiply $X_i$ on both sides of $Y_i = X_i^\top B \bm w + X_i^\top \bm\delta + \ve_i$ and take expectation, we have $E(X_iY_i) = \Sigma B \bm w+ \Sigma\bm  \delta$. Then the vector $\bm\delta$ can be identified as $\bm \delta = \Sigma^{-1}\{E(X_i Y_i) - \Sigma B\bm w\}$. The proof is completed.

\csection{Some Useful Lemmas of Theorem \ref{Theorem convergence rate}}

To prove the theorem condition, we first introduce several useful lemmas.
\begin{lemma}
	(Generalized Hanson-Wright Inequality) Assume that $X \in \mR^p$ is a mean zero sub-Gaussian random vector with $\|X\|_{\psi_2} \le C_{\psi_2}$. Let $B \in \mR^{p\times K}$ be an arbitrary constant matrix, then we have $E\exp(\lambda^2 \|B^\top X\|^2 )\le \exp(C_1C_{\psi_2}^2\lambda^2\|B\|^2_F)$ for all $|\lambda|\le C_2/(C_{\psi_2}\|B\|_{\text{op}})$, where $C_1>0$ and $C_2>0$ are some fixed positive constants.  \label{Lemma hanson wright}
\end{lemma}
\begin{proof}
	This lemma comes from Exercise 6.2.6 in \cite{vershynin2018high}. For the sake of completeness, we provide here a brief proof for the lemma. The proof is based on the results of \cite{hsu2012tail}. Define $A = (B, \mathbf 0_{p\times(p-K)}) \in \mR^{p\times p}$. Then $\|A^\top X\|^2 = X^\top AA^\top X = X^\top B B^\top X = \|B^\top X\|^2$, $\|A\|_{\text{op}} = \|B\|_{\text{op}}$ and $\|A\|_F = \|B\|_F$. As a result, the inequalities in \cite{hsu2012tail} holds for matrix $B\in\mR^{p\times K}$. Then by condition (C2), we know that $E\exp(u^\top X_i )\le \exp(C_3C_{\psi_2}^2\|u\|^2/2)$ for some constant $C_3>0$. Next by Remark (2.3) of \cite{hsu2012tail}, for $|\lambda| < 1/(\sqrt {2C_3} C_{\psi_2}\|B\|_{\text{op}})$, we have \[
	E\exp(\lambda^2 \|B^\top X\|^2) \le \exp\bigg\{C_3 C_{\psi_2}^2 \|B\|_F^2 \lambda^2 + \frac{C_3^2 C_{\psi_2}^4 \lambda^4\|B\|_{\text{op}}^2\|B\|_F^2}{ 1-2C_3C_{\psi_2}^2\|B\|_{\text{op}}^2\lambda^2}\bigg\}.
	\]
	When $|\lambda|\le 1/(\sqrt {3C_3} C_{\psi_2}\|B\|_{\text{op}})$, we then have $C_3^2C_{\psi_2}^4\|B\|_{\text{op}}^2\lambda^4 / (1-2C_3C_{\psi_2}^2\|B\|_{\text{op}}^2\lambda^2) \le C_3 C_{\psi_2}^2\lambda^2$. As a result, we have $E\exp(\lambda^2 \|BX\|^2) \le \exp(2C_3 C_{\psi_2}^2\lambda^2 \|B\|_F^2)$. Take $C_1 = 2C_3$ and $C_2 = 1/\sqrt{3C_3}$, then the lemma proof is complete.\\
\end{proof}

\begin{lemma}
	(Convergence Rate of Sample Covariance Matrix) Let $Z_1,Z_2,\dots\in\mR^K$ be a sequence of independent and identically distributed zero-mean sub-Gaussian random vectors. Suppose there exists some fixed constant $M\ge 1$ such that $\|u^\top Z_i\|_{\psi_2} \le ME(u^\top Z_i)^2$ for all $u\in\mR^p$. Define $\wh \Sigma_Z = n^{-1}\sum_{i=1}^n Z_iZ_i^\top$ and $\Sigma_Z = E(Z_iZ_i^\top)$. Then for every positive integer $n$, we have \[
	\|\wh \Sigma_Z - \Sigma_Z\|_{\text{op}}\le CM^2\Bigg(\sqrt{\frac{K + t}{n}} + \frac{K+t}{n}\Bigg)\|\Sigma_Z\|_{\text{op}}
	\]
	with probability at least $1-2\exp(-t)$ for some constant $C>0$. \label{Lemma sample cov}
\end{lemma}
\begin{proof}
	This result comes from Exercise 4.7.3 of \cite{vershynin2018high}. For the sake of completeness, we provide a here a brief proof. Let $U_i = \Sigma_Z^{-1/2}Z_i$. Then by lemma condition, we have $E(U_i) = 0$, $\cov(U_i) = I_K$ and $\|U_i\|_{\psi_2}\le M$. Define $\wh \Sigma_u = n^{-1}\sum_{i=1}^n U_i U_i^\top$ Then by Theorem 4.6.1 of \cite{vershynin2018high}, we have $\|\wh \Sigma_u - I_K\|_{\text{op}}\le M^2\max\{\delta,\delta^2 \}$ with probability $1-2\exp(-u)$, where $\delta = C_0(\sqrt{K/n} + \sqrt{u/n})$ for some constant $C_0>0$. Let $C = \max\{\sqrt 2C_0, 2C_0^2\}$, we have $\|\wh \Sigma_u - I_K\|_{\text{op}}\le CM^2\{\sqrt{K+u/n} + (K+u)/n\}$ with probability $1-2\exp(-u)$. Then $\|\wh \Sigma_Z - \Sigma_Z\|_{\text{op}} \le \|\wh \Sigma_u - I_K\|_{\text{op}} \|\Sigma_Z\|_{\text{op}} \le CM^2\{\sqrt{K+u/n} + (K+u)/n\}\|\Sigma_Z\|_{\text{op}}$ with probability $1-2\exp(-u)$. \\
\end{proof}

\begin{lemma}
	(Restricted Strong Convexity for Correlated Sub-Gaussian Features) Let $X_1,X_2,\dots, X_n \in \mR^p$ be a sequence of independent and identically distributed zero-mean sub-Gaussian random vectors with $E(X_iX_i^\top ) = \Sigma$ and $\|X_i\|_{\psi_2}\le C_{\psi_2}$. Define $\wh \Sigma = n^{-1}\sum_{i=1}^nX_iX_i^\top$. If $n \ge C_1 \max\{C_{\psi_2}^4/\lambda_{\min}(\Sigma), 1\}\log p$ for some constant $C_1>0$. Then for any $u\in \mR^p$ there exists a universal constant $C>0$ such that
	\[
	u^\top \wh \Sigma u\ge \frac{\lambda_{\min}(\Sigma)}{2} \|u\|^2 - C\lambda_{\min}^2(\Sigma)\max\bigg\{\frac{C_{\psi_2}^4}{\lambda_{\min}^2(\Sigma)},1\bigg\}\bigg(\frac{\log p}{n}\bigg)\|u\|_1^2
	\]
	with probability at least $1-2\exp(-\log p)$. \label{Lemma rsc}
\end{lemma}
\begin{proof}
	The proof of the lemma could be derived by Lemma 1 of \cite{loh2012high} and be found in Lemma B.2 of \cite{zhu2018sparse}.
\end{proof}

\begin{lemma}
	(Convergence Rate of $\wh {\bm w}_{\text{ptl}}$) Under the same conditions as in Theorem \ref{Theorem convergence rate}, then there exist constants $C_0$, $C_1, C_2$ and $C_3$ such that for any $t \le C_0 n$, we have $P(\|\wh {\bm w}_{\text{ptl}} -\bm w\|^2 \le C_1\sum_{k=1}^K s_{\max} \log p/n_k + C_2(K+t)/n ) \ge 1- \exp(-C_3t)$. \label{Lemma rate w}
\end{lemma}
\begin{proof}
	Recall that $\wh {\bm w}_{\text{ptl}} = \argmin_{\bm w} \mL_w(\bm w, \wh Z)$. Then by standard ordinary least squares, we have
	\beqrs
	\wh {\bm w}_{\text{ptl}}  &=& \Big(n^{-1}\sum_{i=1}^n \wh Z_i \wh Z_i^\top\Big)^{-1} \Big\{n^{-1}\sum_{i=1}^n \wh Z_i(Y_i - \wh Z_i^\top w)\Big\}\\
	&=& \bm w + \wh \Sigma_{\wh Z}^{-1}(Q_1 + Q_2+Q_3),
	\eeqrs
	where $\wh \Sigma_{\wh Z} = n^{-1}\sum_{i=1}^n \wh Z_i \wh Z_i^\top$, $Q_1 = n^{-1}\sum_{i=1}^n \wh Z_i (Z_i - \wh Z_i)^\top \bm w$, $Q_2 = n^{-1}\sum_{i=1}^n \wh Z_i X_i^\top \bm \delta$ and $Q_3 = n^{-1}\sum_{i=1}^n \wh Z_i \ve_i$. We next evaluated the quantities $\wh \Sigma_{\wh Z}$, $Q_1$, $Q_2$ and $Q_3$ subsequently. The desired conclusion follows if we can show the following conclusions hold with probability at least $1-\exp(-C_3t)$: (1) $\lambda_{\min}(\wh\Sigma_{\wh Z})\ge \tau_{\min}$; (2) $\|Q_1\| \le \tau_{\max}(\sum_{k=1}^Ks_{\max} \log p/n_k )^{1/2}$; (3) $\|Q_2\| \le \tau_{\max}h[(\sum_{k=1}^Ks_{\max} \log p/n_k )^{1/2}+\{(K+t)/n\}^{1/2}]$; and (4) $\|Q_3\| \le \tau_{\max}\{(K+ t)/n\}^{1/2}$ for some positive constants $\tau_{\max}, \tau_{\min} >0$.
	
	\textit{Step 1.} We first consider the term $\wh\Sigma_{\wh Z} = \Sigma_Z + \Delta_{z1} + \Delta_{z2} + \Delta_{z2}^\top + \Delta_{z3} $, where $\Sigma_Z = E(Z_iZ_i^\top) = B^\top \Sigma B$, $\Delta_{z1} =  n^{-1}\sum_{i=1}^n(\wh Z_i - Z_i)(\wh Z_i - Z_i)^\top$, $\Delta_{z2} = n^{-1}\sum_{i=1}^n Z_i(\wh Z_i - Z_i)^\top$ and $\Delta_{z3} = n^{-1}\sum_{i=1}^n Z_iZ_i^\top - \Sigma_Z$. Then we next show that $\|\Delta_{zj}\|_{\text{op}} = o(1)$ for $j = 1,2,3$ with high probability as follows.
	
	\textit{Step 1.1.} We begin with $\Delta_{z1}$. By triangle inequality, we have
	$\|\Delta_{z1}\|_F^2 \le n^{-1}\sum_{i=1}^n \|\wh Z_i - Z_i \|^2 = n^{-1}\sum_{i=1}^n \|(\wh B - B)^\top X_i\|^2$. 
	Note that $\wh B - B$ is independent with $X_i$. Then by Lemma \ref{Lemma hanson wright}, we know that conditional on $\wh B - B$, $\|(\wh B - B)X_i\|^2$ is sub-exponential with $\Big\| \|(\wh B - B)^\top X_i\|^2\Big\|_{\psi_1}\le C_1\|\wh B - B\|_F^2\|X_i\|_{\psi_2}^2 \le C_1C_{\psi_2}^2\|\wh B - B\|_F^2$ for some constant $C_1>0$, where $\|U\|_{\psi_1} = \inf\{t>0: E\exp(|U|/t)\le 2\}$ is the sub-exponential norm of a sub-exponential random variable. Then by Bernstein's inequality \citep{vershynin2018high} and condition (C2), we have
	\[
	P\bigg(n^{-1}\sum_{i=1}^n \Big\|(\wh B - B)X_i\Big\|^2 - \tr\Big\{(\wh B - B)\Sigma (\wh B - B)^\top\Big\} \ge C_1C_{\psi_2}^2 \|\wh B - B \|_F^2 \bigg) \le \exp(-c_1n)
	\]
	for some constant $c_1>0$. Note that if $A\le B$, then $\tr(A)\le \tr(B)$ for two arbitrary symmetric matrices $A,B\in\mR^{p\times p}$ \citep{seber2008matrix}. Then by condition (C4), we have $\tr\{(\wh B - B)\Sigma (\wh B - B)^\top\} \le C_{\max} \|\wh B - B\|_F^2$. We have $P(n^{-1}\sum_{i=1}^n \|(\wh B - B)X_i\|^2 \ge C_2 \|\wh B - B \|_F^2) \le \exp(-c_1 n) $, where $C_2 = C_{\max} + C_1C_{\psi_2}^2$.
	By condition (C4), we can show that $P(\|\wh B - B\|_F^2 \ge C_3\sum_{k=1}^Ks_{\max } \log p/n_k )\le \sum_{k=1}^K P(\|\wh{\bm \beta}^{(k)}-\bm  \beta^{(k)}\|^2 \ge C_3s_{\max }\log p/n_k) \le \exp (-c_1\log p)$ for some $C_3>0$. Then we know that
	\beqrs
	& &P\bigg(n^{-1}\sum_{i=1}^n \Big\|(\wh B - B)X_i\Big\|^2\ge C_4 \sum_{k=1}^K\frac{s_{\max}\log p}{n_k} \bigg)  \\
	&=& E\bigg\{P\Big(n^{-1}\sum_{i=1}^n \Big\|(\wh B - B)X_i\Big\|^2\ge C_4 \sum_{k=1}^K\frac{s_{\max}\log p}{n_k} \Big| \wh B - B \Big) \bigg\} \\
	&\le& P\bigg(\|\wh B - B\|_F^2 \ge C_3 \sum_{k=1}^K\frac{s_{\max}\log p}{n_k}\bigg) \\
	&+& E\bigg\{P\Big(n^{-1}\sum_{i=1}^n \Big\|(\wh B - B)X_i\Big\|^2\ge C_2\|\wh B - B\|_F^2\Big| \wh B - B \Big)\bigg\}\\
	&\le& 2\exp(-c_1\log p),
	\eeqrs
	where $C_4 = C_2C_3$ and the last inequality holds by $n \ge \log p$. As a result, we have $P(\|\Delta_{z1}\|_F\ge C_4\sum_{k=1}^Ks_{\max} \log p/n_k ) \le 2\exp(-c_1\log p)$.
	
	\textit{Step 1.2. }Then we consider $\Delta_{z3}$.  To apply Lemma \ref{Lemma sample cov}, we need to bound $\|Z_i\|_{\psi_2}$, $M$ and $\|\Sigma_Z\|_{\text{op}}$. Note that $\|Z_i\|_{\psi_2} = \sup_{\|u\| = 1} \|u^\top Z_i\|_{\psi_2} = \sup_{\|u\| = 1}\|u^\top B^\top X_i\|_{\psi_2} \le \sup_{\|u\| = 1}\|Bu\|\|X_i\|_{\psi_2}  = \|B\|\|X_i\|_{\psi_2}\le C_{\psi_2}C_{\max}^{1/2}$. As a result, $Z_i$ is sub-Gaussian with $\|Z_i\|_{\psi_2} \le C_{\psi_2}C_{\max}^{1/2}$.   Without loss of generality, we next obtain an upper bound for $M$ when the vector $u \in\mR^p $ satisfies $\|u\| =1$ (otherwise consider $u_0 = u / \|u\|$). By condition (C3) and (C4), we have $\|u^\top X_i\|_{\psi_2}^2\le \|X_i\|_{\psi_2}^2\le C_{\psi_2}^2\le C_{\psi_2}^2C_{\min}^{-1} \lambda_{\min}(\Sigma) \le  C_{\psi_2}^2C_{\min}^{-1}u^\top \Sigma u = C_{\psi_2}^2C_{\min}^{-1}E(u^\top X_i)^2$. 
	Then we know that $\|u^\top Z_i\|_{\psi_2}^2 = \|(Bu)^\top X_i\|_{\psi_2}^2 \le C_{\psi_2}^2C_{\min}^{-1} E(u^\top B^\top X_i)^2 = C_{\psi_2}^2C_{\min}^{-1} E(u^\top Z_i)^2$. As a result, we know that $M \le C_{\psi}^2 C_{\min}^{-1}$. Note that $\|\Sigma_Z\|_{\text{op}} = \|B^\top \Sigma B\|_{\text{op}}\le C_{\max }\|B\|_{\text{op}}^2\le C_{\max}^2$. Then by Lemma \ref{Lemma sample cov}, we have
	\[
	P\bigg(\|\wh \Sigma_Z - \Sigma_Z\|_{\text{op}} \ge C_5 \Big(\frac{K + \log p}{n}\Big)^{1/2}\bigg) \le \exp(-c_1\log p)
	\]
	for some constant $C_5 >0$. As a result, we know that $P(\|\Delta_{z3}\|_{\text{op}} \ge C_5\{(K+\log p) / n\}^{1/2}) \le \exp(-c_1\log p)$.
	
	\textit{Step 1.3. }We next study $\Delta_{z2}$. By Cauchy-Schwarz inequality and note that $\| A\| = \sup_{\|u\| = 1,\| v \| = 1} u^\top A v$ for an arbitrary matrix $A$ \citep{vershynin2018high}, we can obtain the following inequality
	\beqrs
	\|\Delta_{z2}\|_{\text{op}} &=& \Big\|n^{-1}\sum_{i=1}^n Z_i(\wh Z_i - Z_i)^\top \Big\|_{\text{op}}\\
	&= &\sup_{\|u\| = 1,\| v \| = 1} n^{-1}\sum_{i=1}^n u^\top Z_i(\wh Z_i - Z_i)^\top v\\
	&\le& \sup_{\|u\| = 1,\| v \| = 1} \bigg(n^{-1}\sum_{i=1}^n u^\top Z_iZ_i^\top u\bigg)^{1/2} \bigg(n^{-1}\sum_{i=1}^n v^\top (\wh Z_i - Z_i)(\wh Z_i - Z_i)^\top v\bigg)^{1/2}\\
	&\le& \Big\|n^{-1}\sum_{i=1}^n Z_iZ_i^\top \Big\|_{\text{op}}^{1/2} \bigg(n^{-1}\sum_{i=1}^n \| \wh Z_i - Z_i\|^2 \bigg)^{1/2}\\
	&\le& \Big(\|\wh \Sigma_Z - \Sigma_Z\|_{\text{op}} + \|\Sigma_Z\|_{\text{op}} \Big)^{1/2}\bigg(n^{-1}\sum_{i=1}^n \| \wh Z_i - Z_i\|^2 \bigg)^{1/2} \\
	&\le& \bigg\{C_5 \Big(\frac{K + \log p}{n} \Big)^{1/2}+ C_{\max}^2 \bigg\}^{1/2}\bigg(C_4\sum_{k=1}^K\frac{s_{\max} \log p}{n_k}\bigg)^{1/2}
	\eeqrs
	with probability $1-3\exp(-c_1\log p)$. By condition (C2), we know that $P(\|\Delta_{z3}\|_{\text{op}}\ge C_6(\sum_{k=1}^K s_{\max}\log p/n_k)^{1/2}) \le \exp(-c_1\log p)$ for some constant $C_6 >0$.
	
	By condition (C4), we have $\lambda_{\min}(\Sigma_Z)  = \lambda_{\min}(B^\top \Sigma B) = \inf_{\|u\| = 1} u^\top B^\top \Sigma B u \ge C_{\min}\inf_{\|u\| = 1}u^\top B^\top Bu \ge C_{\min}^2$. Combining the above results and by condition (C2), we have $\lambda_{\min}(\wh \Sigma_{\wh Z}) \ge \lambda_{\min}(\Sigma_Z) - \|\Delta_{z1}\|_F - 2\|\Delta_{z2}\|_{\text{op}} - \|\Delta_{z3}\|_{\text{op}} \ge C_{\min}/2$ with probability at least $1- \exp(-c_2\log p)$ for some $c_2 >0$.
	
	\textit{Step 2. }Then we consider the term $Q_1$. We can decompose it into $Q_1 = Q_{11} + Q_{12}$, where $Q_{11} = n^{-1}\sum_{i=1}^n (\wh Z_i - Z_i)(Z_i - \wh Z_i)^\top \bm w$ and $Q_{12} = n^{-1}\sum_{i=1}^n Z_i (Z_i - \wh Z_i)^\top \bm w$. By the proof in Step 1.1, we know that  $\|Q_{11}\| = \sup_{\|u\| = 1} u^\top \Delta_{z1} \bm w \le \|\Delta_{z1}\|_{\text{op}} \|\bm w\| \le \|\Delta_{z1}\|_{F} \|\bm w\| \le  C_4C_w\sum_{k=1}^Ks_{\max} \log p/n_k $ and $\|Q_{12}\| = \sup_{\|u\| = 1} u^\top \Delta_{z2}\bm w \le \|\Delta_{z2}\|_{\text{op}} \|\bm w\|  \le C_6C_w(\sum_{k=1}^Ks_{\max} \log p/n_k)^{1/2}$ with probability at least $1-2\exp(-c_2\log p)$.
	
	\textit{Step 3. }We next consider the term $Q_2$. Similarly, it can be decomposed as $Q_2 = Q_{21} + Q_{22}$, where $Q_{21} = n^{-1}\sum_{i=1}^n (\wh Z_i - Z_i)X_i^\top \bm \delta$ and $Q_{22} = n^{-1}\sum_{i=1}^n Z_i X_i^\top\bm \delta$. Then we consider the two terms subsequently.
	
	\textit{Step 3.1. }We first consider the term $Q_{21}$. By Cauchy-Schwarz inequality, we have
	\beqrs
	\|Q_{21}\| &=& \Big\| n^{-1}\sum_{i=1}^n (\wh B - B)^\top X_i X_i^\top \bm \delta\Big\| \\
	&\le &\Big\|n^{-1}\sum_{i=1}^n (\wh B - B )^\top X_iX_i^\top (\wh B -B)\Big\|_{\text{op}} \Big(n^{-1}\sum_{i=1}^n\bm\delta^\top X_iX_i^\top \bm \delta\Big)^{1/2}.
	\eeqrs
	By Step 1.1.1, we know that $\|n^{-1}\sum_{i=1}^n (\wh B - B )^\top X_iX_i^\top (\wh B -B)\|_{\text{op}} \le n^{-1}\sum_{i=1}^n \|X_i^\top (\wh B - B)\|_{\text{op}}^2 \le C_4\sum_{k=1}^Ks_{\max} \log p/n_k $ with probability at least $1-2\exp(-c_1\log p)$. Note that $\bm \delta^\top X_i$ is sub-Gaussian with $\|\bm \delta^\top X_i\|_{\psi_2} \le \|\bm \delta\| \|X_i\|_{\psi_2} \le C_{\psi_2}\|\bm \delta\|_1 \le C_{\psi_2}h$. As a result, $(\bm \delta^\top X_i)^2$ is sub-exponential with $\|(\bm \delta^\top X_i)^2\|_{\psi_1} = \|\bm \delta^\top X_i\|_{\psi_2}^2 \le C_{\psi_2}^2h^2 $. Then by Bernstein's inequality \citep{vershynin2018high} and condition (C2), we have
	\[
	P\Big(\Big|n^{-1}\sum_{i=1}^n \bm \delta^\top X_iX_i^\top \bm \delta - \bm \delta^\top \Sigma \bm \delta \Big| \ge C_6C_{\psi_2}^2n^{-1}h^2\log p\Big) \le 2\exp(-c_1\log p)
	\]
	for some constant $C_6>0$. Note that $\bm \delta^\top \Sigma\bm \delta \le C_{\max} \|\bm \delta\|^2 \le C_{\max}h^2$. Then we know that $n^{-1}\sum_{i=1}^n \bm \delta^\top X_iX_i^\top \bm \delta \le h^2(C_{\max} + C_6C_{\psi_2}^2n^{-1}\log p)$ with probability at least $1- 2\exp(-c_1\log p)$. As a result, $\|Q_{21}\|\le C_7h(\sum_{k=1}^Ks_{\max} \log p/n_k )^{1/2}$ with probability at least $1-\exp(-c_1\log p)$ for some constant $C_7 \ge C_4(C_{\max} + C_6C_{\psi_2}^2n^{-1}
	\log p)^{1/2}$.
	
	\textit{Step 3.2. }Then we consider the term $Q_{22}$. We apply the $\ve$-net technique here. For the unit sphere $\mathbb S^{K-1} = \{u\in\mR^K: \|u\| = 1\}$. There exists a 1/2-net $\mN$ such that $|\mN|\le 5^K$. Let $u_0\in\mathbb S^{K-1}$ satisfying $u_0^\top Q_{22} = \|Q_{22}\|$. Then there exists a $u\in\mN$ such that $\|u - u_0 \| \le 1/2$. Then we have $\|Q_{22}\| = |u_0^\top Q_{22}|\le |(u_0 - u)^\top Q_{22}|+|u^\top Q_{22} |\le 2^{-1}\|Q_{22}\| + |u^\top Q_{22}|$. As a result, we have $\|Q_{22}\| \le 2\sup_{u\in\mN} |u^\top Q_{22}|$. Then for a fixed $u\in\mathbb S^{n-1}$, consider $u^\top Q_{22} =  n^{-1}\sum_{i=1}^n u^\top Z_iX_i^\top \bm \delta$. Note that $X_i^\top \bm \delta$ is sub-Gaussian with $\|X_i^\top \bm \delta\|_{\psi_2} \le C_{\psi_2}h$. For all $u\in \mR^K$ satisfying $\|u\|=1$, $u^\top Z_i$ is sub-Gaussian with $\|u^\top Z_i\|_{\psi_2} \le \|Z_i\|_{\psi_2}\le C_{\psi_2}C_{\max}^{1/2}$. As a result $u^\top Z_i X_i^\top\bm \delta$ is sub-exponential with $\|u^\top Z_i X_i^\top \bm \delta\|_{\psi_1} \le C_{\psi_2}^2 C_{\max}^{1/2}h$. By Bernstein's inequality, for all $\eta>0$, we have
	\[
	P\bigg(\Big|n^{-1}\sum_{i=1}^n u^\top Z_i X_i^\top \bm \delta\bigg|\ge \eta hC_{\psi_2} C_{\max}^{1/2}\bigg)\le 2\exp\Big\{-cn\min(\eta,\eta^2)\Big\}.
	\]
	Then we know that $P(\|Q_{22}\|\ge \eta hC_{\psi_2} C_{\max}^{1/2})
	\le P(\sup_{u\in\mN} |n^{-1}\sum_{i=1}^n u^\top Z_i X_i^\top \bm \delta|\ge \eta hC_{\psi_2} C_{\max}^{1/2})\le 2\times 5^K\exp\{-cn\min(\eta,\eta^2)\}$. Take $\eta = C_8\{(K+t)/n\}^{1/2}<1$ for some large enough $C_8>0$ such that we have $
	P(\|Q_{22}\|\ge C_8h\{(K+ t)/n\}^{1/2})\le \exp(-c_1t)$.

	\textit{Step 4. } Lastly we consider $Q_3 = Q_{31} + Q_{32}$ where $Q_{31} = n^{-1}\sum_{i=1}^n (\wh Z_i - Z_i)\ve_i$ and $Q_{32} = n^{-1}\sum_{i=1}^n Z_i \ve_i$. Then we evaluate $Q_{31}$ and $Q_{32}$ subsequently.
	
	\textit{Step 4.1.} We first consider $Q_{31}$. Similar to the proof in Step 1.2, we know that conditional on $\wh B - B$, $(\wh Z_i - Z_i) = (\wh B - B)^\top X_i$ is sub-Gaussian with $\| \wh Z_i - Z_i\|_{\psi_2} \le \|\wh B - B\| \|X_i\|_{\psi_2}$. Note that $\ve_i$ is sub-Gaussian. By a similar proof of Step 3.2, we have
	\[
	P\bigg(\|Q_{31}\|\ge C_9\Big(\frac{K+\log p}n\Big)^{1/2}\|\wh B - B\|_{\text{op}}\Big| \wh B - B\bigg) \le \exp(-c_1\log p)
	\]
	for some constant $C_9>0$. By Step 1.1, we know that $\|\wh B - B\|_{\text{op}} \le \|\wh B - B\|_F \le C_4 \sum_{k=1}^K s_{\max} \log p /n_k$ with probability $1-2\exp(-c\log p)$. Then by a similar procedure of Step 1.1, we can show that
	\[
	P\Bigg(\|Q_{31}\|\ge C_{10}\bigg(\frac{K + \log p}n\sum_{k=1}^K\frac{s_{\max }\log p}{n_k}\bigg)^{1/2}\Bigg) \le \exp(-c_1\log p	 ).
	\]
	
	\textit{Step 4.2. }Lastly, we consider $Q_{32}$. Recall that $Z_i$ is sub-Gaussian with $\|Z_i\|_{\psi_2} \le LC_{\max}^{1/2}$. Then by a similar proof of Step 3.2, we have $ P(\|Q_{32}\| \ge C_{11}\{(K+ t)/n\} ^{1/2}) \le \exp(-c_1t)$.
	Recall that we have shown that $\lambda_{\min}(\wh\Sigma_{\wh Z}) \ge C_{\min} / 2$ with probability at least $1-\exp(-c_2\log p)$. Combining the above results and by the Cauchy-Schwarz inequality, we know that
	\beqrs
	\|\wh {\bm w}_{\text{ptl}} - \bm w\|^2 &\le& 2C_{\min}^{-1} (\wh {\bm w}_{\text{ptl}} - \bm w)^\top\wh \Sigma_{\wh Z} (\wh {\bm w}_{\text{ptl}} - \bm w) \\&\le & 2C_{\min}^{-1}\Big\|(\wh {\bm w}_{\text{ptl}} - \bm w)^\top\Big(Q_1+Q_2+Q_3\Big)\Big\|\\
	&\le&2C_{\min}^{-1} \|\wh {\bm w}_{\text{ptl}} - \bm w\|\Big(\|Q_1\| + \|Q_2\| + \|Q_3\|\Big)\\ &\le& C_{12}\|\wh {\bm w}_{\text{ptl}} - \bm w\| \bigg\{\Big(\sum_{k=1}^K \frac{s_{\max}\log p}{n_k}\Big)^{1/2} + \Big(\frac{K + t}{n}\Big)^{1/2} \bigg\}
	\eeqrs
	for some large enough constant $C_{12} >0$ with probability at least $1-\exp\{-(c_1+1)\log (n \wedge p)\}$. As a result, we have $P(\|\wh {\bm w}_{\text{ptl}} - \bm w\| \le C_{12}[(\sum_{k=1}^K s_{\max}\log p/n_k)^{1/2} + \{(K+t)/n\}^{1/2} ]) \ge 1-\exp(-c_1t)$. \\
\end{proof}

\begin{lemma}
	(Convergence Rate of $\wh{\bm \delta}_{\text{ptl}}$) Under the same conditions as in Theorem \ref{Theorem convergence rate}, there exist some positive constants $C_1$ and $C_2$ such that $P(\|\wh {\bm \delta}_{\text{ptl}} - \bm \delta\|^2 \le C_1(s_{\delta}r_n^2\wedge hr_n\wedge h^2))\ge 1-\exp\{-C_2\log (n\wedge p)\}$. \label{Lemma rate delta}
\end{lemma}
\begin{proof}
	Define $\wh {\bm u}_{\delta} = \wh {\bm \delta}_{\text{ptl}} - \bm \delta$. Note that $\wh {\bm \delta}_{\text{ptl}} = \argmin_{\bm \delta} \mL_{\text{target},\lambda_\delta}(\bm \delta)$. Then we know that $ \mL_{\text{target},\lambda_\delta}(\wh {\bm\delta}_{\text{ptl}}) \le  \mL_{\text{target},\lambda_\delta}( \bm \delta)$. We then obtain the following relationship
	\beqrs
	\frac 12 \wh {\bm u}_{\delta}^\top \wh \Sigma \wh {\bm u}_{\delta} &\le& \lambda_\delta \|\bm \delta\|_1 - \lambda_\delta \|\wh {\bm \delta}_{\text{ptl}}\|_1 + \Big| \wh {\bm u}_\delta^\top n^{-1}\sum_{i=1}^n X_i\Big(\wh e_i - X_i^\top\bm \delta\Big)\Big|\\
	&\le &\lambda_\delta \|\bm \delta\|_1 - \lambda_\delta \|\wh {\bm \delta}_{\text{ptl}}\|_1 + \|\wh {\bm u}_\delta\|_1 \Big\| n^{-1}\sum_{i=1}^n X_i\Big(\wh e_i - X_i^\top\bm \delta\Big)\Big\|_\infty,
	\eeqrs
	where $\wh \Sigma = n^{-1}\sum_{i=1}^n X_iX_i^\top$. 
	By Lemma \ref{Lemma rsc} and condition (C4), we have
	\beqr
	\wh {\bm u}_{\delta}^\top \wh \Sigma\wh {\bm u}_{\delta}&\ge& \frac{\lambda_{\min}(\Sigma)}{2} \|\wh {\bm u}_{\delta}\|^2 - C_{11}\lambda_{\min}(\Sigma)\max\bigg\{\frac{C_{\psi_2}^4}{\lambda_{\min}^2(\Sigma)},1\bigg\}\frac{\log p}{n}\|\wh {\bm u}_{\delta}\|_1^2\nonumber \\
	&\ge &\frac{C_{\min}}{2} \|\wh {\bm u}_{\delta}\|^2 - C_{12}\frac{\log p}{n}\|\wh {\bm u}_{\delta}\|_1^2\label{equation rsc for sigmaxx}
	\eeqr
	with $C_{12} = C_{11} \max\{C_{\psi_2}^4C_{\min}^{-1}, C_{\max}\}$ for some constant $C_{11} >0$ with probability $1-2\exp(-c_1\log p)$.
	Then we are going to show $\| n^{-1}\sum_{i=1}^n X_i(\wh e_i - X_i^\top\bm \delta)\|_\infty\le \lambda_w$ with high probability. Recall that $\wh e_i = Y_i - \wh Z_i^\top \wh w_{\text{ptl}}$ and $Y_i = Z_i^\top \bm w + X_i^\top \bm \delta + \ve_i$. Then we have
	\beqrs
	\Big\|n^{-1}\sum_{i=1}^n X_i(\wh e_i - X_i^\top \bm \delta)\Big\|_{\infty} &=& \Big\|n^{-1}\sum_{i=1}^n X_i(Z_i^\top \bm w - \wh Z_i^\top \wh {\bm w}_{\text{ptl}} + \ve_i)\Big\|_{\infty}\\
	&\le &\Delta_1 + \Delta_2 + \Delta_3 
	\eeqrs
	where $\Delta_1 = \|n^{-1}\sum_{i=1}^n X_i(Z_i - \wh Z_i)^\top \bm w\|_{\infty}$,
	$\Delta_2 = \|n^{-1}\sum_{i=1}^n X_i \wh Z_i(\bm w - \wh {\bm w}_{\text{ptl}})\|_{\infty}$ and 
	$\Delta_3 = \|n^{-1}\sum_{i=1}^n X_i\ve_i\|_{\infty}$. Then we next show that $P(\Delta_j > Cr_n) \le \exp(-c_1\log p)$ for $j = 1,2,3$.
	
	\textsc{Step 1.} We start with $\Delta_1$. Note that for $1\le j\le p$, conditional on $\wh B-B$, $X_{ij}(Z_i - \wh Z_i)^\top \bm w$ is sub-exponential with $\|X_{ij} (Z_i - \wh Z_i)^\top \bm w\|_{\psi_1} \le \|X_{ij}\|_{\psi_2} \|(Z_i - \wh Z_i)^\top \bm w\|_{\psi_2} \le C_{\psi_2}^2 \|\bm w\| \|\wh B - B\|_{\text{op}}$. Note that by Cauchy-Schwarz inequality, we have $E\{X_{ij}(Z_i - \wh Z_i)^\top w|\wh B - B\} = E\{e_j^\top X_iX_i^\top (\wh B - B)^\top w|\wh B - B\}= e_j^\top \Sigma(\wh B - B)^\top \bm w \le \lambda_{\max}(\Sigma) \|(\wh B - B)^\top \bm w\| \le C_{\max}\| \wh B - B\|_{\text{op}} \|\bm w\|$, where $e_j = (0,\dots,1,\dots,0)^\top$ is the canonical basis with the $j$th component equal to 1 otherwise 0. Then by Bernstein's inequality, we know that
	\begin{gather*}
		P\bigg(\bigg|n^{-1}\sum_{i=1}^nX_{ij}(Z_i - \wh Z_i)^\top \bm w - E\Big\{X_{ij}(Z_i - \wh Z_i)^\top \bm w \Big|\wh B - B\Big\}\bigg |\\ \quad\quad \quad \quad \quad  \quad \quad \quad \ge C_{13}C_{\psi_2}^2 \|\bm w\| \|\wh B - B\|_{\text{op}} \Big(\frac{\log p}n\Big)^{1/2}\bigg |\wh B - B\bigg) \le 2\exp(-c_1\log p).
	\end{gather*}
	As a result, we have
	\begin{gather*}
		P\bigg(\Big|n^{-1}\sum_{i=1}^nX_{ij}(Z_i - \wh Z_i)^\top \bm w\Big| \ge C_{14} \|\bm w\| \Big\|\wh B - B\Big\|_{\text{op}}\bigg|\wh B - B \bigg)\le 2\exp\Big\{-(c_1+1)\log p\Big\}
	\end{gather*}
	for $C_{14} \ge C_{13}C_{\psi_2}^2(\log p/n)^{1/2} + C_{\max}$. Then by union bound, we know that
	\beqrs
	& &P\bigg(\Delta_1 \ge C_{14} \|\bm w\| \Big\|\wh B - B\Big\|_{\text{op}} \bigg|\wh B - B\bigg)\\
	&\le &\sum_{j=1}^pP\bigg(\Big|n^{-1}\sum_{i=1}^nX_{ij}\Big(Z_i - \wh Z_i\Big)^\top \bm w\Big| \ge C_{14} \|\bm w\| \Big\|\wh B - B\Big\|_{\text{op}}\bigg|\wh B - B \bigg)\\&\le& 2\exp(-c_1\log p).
	\eeqrs
	Note that $\|\wh B - B\|_{\text{op}}\le \|\wh B -B\|_F \le C_3(\sum_{k=1}^K s_{\max}\log p/n_k)^{1/2}$ with probability $1-\exp(-c_1\log p)$. Then there exists a constant $C_{15} \ge \|w\|C_3C_{14}$ such that
	\beqrs
	P\bigg(\Delta_1 \ge C_{15}\Big(\sum_{k=1}^K\frac{s_{\max}\log p}{n_k}\Big)^{1/2}\bigg)
	&\le&  P\bigg(\|\wh B - B\|_F^2 \ge C_3\sum_{k=1}^K\frac{s_{\max}\log p}{n_k}\bigg) \\&+& E\bigg\{P\Big(\Delta_1 \ge C_{14} \|\bm w\| \Big\|\wh B - B\Big\|_{\text{op}} \bigg|\wh B - B\Big)\bigg\}\\&\le& 3\exp(-c_1\log p).
	\eeqrs
	By condition (C2), we know that $\sum_{k=1}^Ks_{\max}\log p/ n_k = o(1)$. Then we know that $P(\Delta_1 \le C_{15}(\sum_{k=1}^K s_{\max}\log p/n_k)^{1/2})\ge 1-3\exp(-c_1\log p)$.
	
	\textsc{Step 2.} Then we consider $\Delta_2$. By Cauchy Schwarz inequality, we have
	\beq
	\Big|n^{-1}\sum_{j=1}^p X_{ij} \wh Z_i^\top(\bm w - \wh {\bm w}_{\text{ptl}})\Big| \le \bigg(n^{-1}\sum_{i=1}^n X_{ij}^2\bigg)^{1/2}\bigg(n^{-1}\sum_{i=1}^n \| \wh Z_i^\top (\wh {\bm w}_{\text{ptl}} - \bm w) \|^2\bigg)^{1/2}. \label{equation cauchy for delta2}
	\eeq
	By condition (C3), we know that $X_{ij}^2$ is sub-exponential with $E(X_{ij}^2 ) \le C_{\max}$ and $\|X_{ij}^2\|_{\psi_1} = \|X_{ij}\|_{\psi_2}^2 \le C_{\psi_2}^2$. Then by Bernstein's inequality, we have { $P(n^{-1}\sum_{i=1}^n X_{ij}^2  - E(X_{ij}^2)\ge C_{16}L^2(\log p/n)^{1/2}) \le \exp\{-(c_1+1)\log p\}$ for some $C_{16}>0$. Note that $E(X_{ij}^2) = O(1)$. As a result, we have $P(n^{-1}\sum_{i=1}^n X_{ij}^2 \ge C_{17}) \le \exp\{-(c_1+1)\log p\}$ for $C_{17} \ge \max_j E(X_{ij}^2) + C_{16}L^2(\log p/n)^{1/2}$. }By Lemma \ref{Lemma rate w}, we know that $n^{-1}\sum_{i=1}^n \|\wh Z_i(\wh {\bm w}_{\text{ptl}} - \bm w)\|^2 = (\wh {\bm w}_{\text{ptl}} - \bm w)^\top \wh \Sigma_{\wh Z} (\wh {\bm w}_{\text{ptl}} - \bm w) \le 2C_{12}^2\{\sum_{k=1}^K s_{\max}\log p/n + (K+\log p)/n\} $ with probability at least $1-\exp\{-(c_1+1)\log p\}$. Then by union bound and \eqref{equation cauchy for delta2}, we have
	\beqrs
	& &P\bigg(\Delta_2 \ge C_{18}\bigg\{\Big(\sum_{k=1}^K \frac{s_{\max}\log p}{n_k}\Big)^{1/2} +  \Big(\frac Kn\Big)^{1/2} + \Big(\frac{\log p}n\Big)^{1/2}\bigg\}\bigg)\\
	&\le& \sum_{j=1}^p P\bigg(\Big|n^{-1}\sum_{i=1}^n X_{ij} \wh Z_i^\top(\bm w - \wh {\bm w}_{\text{ptl}})\Big| \\&&\ge C_{18}\bigg[ \Big(\sum_{k=1}^K \frac{s_{\max}\log p}{n_k}\Big)^{1/2} + \Big(\frac Kn\Big)^{1/2} + \Big(\frac{\log p}n\Big)^{1/2}\bigg]\bigg)\\
	&\le& \sum_{j=1}^p P\Big(n^{-1}\sum_{i=1}^n\|\wh Z_i^\top(\wh {\bm w}_{\text{ptl}} - \bm w)\|^2 \ge 2C_{12}\Big\{\sum_{k=1}^K \frac{s_{\max} \log p}{n_k} + \frac{K+\log p}{n} \Big\}\Big)\\
	&+ &\sum_{j=1}^p P\Big(n^{-1}\sum_{i=1}^n X_{ij}^2 \ge C_{17}\Big) \le 2\exp\{-c_1\log p\},
	\eeqrs
	where $C_{18}^2\ge 2C_{12}C_{17}$. As a result, we have $P(\Delta_2 \le C_{18}[(\sum_{k=1}^K s_{\max }\log p/n_k)^{1/2} + \{(K+\log p)/n\}^{1/2}])\ge 1-2\exp(-c_1\log p)$.

	\textsc{Step 3. }Then we evaluate $\Delta_3$. Note that $X_{ij}\ve_i$ is sub-exponential with $\|X_{ij}\ve_i\|_{\psi_1} \le C_{\psi_2}^2$. By union bound and Bernstein's inequality, we have
	\beqrs
	&&P\bigg(\Big\|n^{-1}\sum_{i=1}^n X_i\ve_i\Big\|_{\infty} \ge C_{19} C_{\psi_2}^2\Big(\frac{\log p}{n}\Big)^{1/2}\bigg)\\
	&\le& \sum_{j=1}^p P\bigg(\Big|n^{-1}\sum_{i=1}^n X_{ij}\ve_i\Big|\ge C_{19} C_{\psi_2}^2\Big(\frac{\log p}{n}\Big)^{1/2}\bigg)
	\le \exp(-c_1\log p)
	\eeqrs
	for some constant $C_{19}>0$. Then we take $\lambda_{\delta}= C_{\delta}r_n$ for some large enough $C_\delta>0$. Then we know that $P(\|n^{-1}\sum_{i=1}^n X_i(\wh e_i - X_i^\top\bm \delta)\|_{\infty} \le \lambda_{\delta}/2) \ge 1-3\exp(-c_1\log p)$.
	
	\textsc{Step 4.} In this step, we combine the above results to show the convergence rate of $\wh \delta_{\text{ptl}}$. Let $E = \{\eqref{equation rsc for sigmaxx}\text{ holds and }\|n^{-1}\sum_{i=1}^n X_i(\wh e_i - X_i^\top\bm \delta)\|_{\infty} \le \lambda_{\delta} /2\}$. Then by Steps 1-3, we know that $P(E) \ge 1-3\exp(-c_1\log p)$ when $\lambda_{\delta} \ge C_{\lambda} r_n$. Let $\bm u_{\mS} = \{u_j:j\in\mS\}$ be the sub-vector of $\bm u$ with the index set $\mS$.
	
	We first show that the $\|\wh {\bm u}_{\delta}\|^2 \lesssim s_{\delta}\lambda_{\delta}^2 $. On the event $E$, we have
	\beqrs
	\frac 12 \wh {\bm u}^\top_{\delta}\wh \Sigma \wh {\bm u}_{\delta} &\le& \lambda_{\delta} \|\bm \delta\|_1 - \lambda_{\delta}\|\wh {\bm \delta}_{\text{ptl}}\|_1 + \frac 12 \lambda_{\delta} \|\wh {\bm u}_{\delta}\|_1 \\
	&=& \lambda_{\delta}\|\bm \delta_{\mS}\|_1 -\lambda_{\delta}\|\wh {\bm \delta}_{\text{ptl},\mS}\|_1 - \lambda_{\delta}\|\wh {\bm \delta}_{\text{ptl},\mS^c}\|_1 + \frac12 \lambda_{\delta}\|\wh{ \bm u}_{\delta,\mS}\|_1 + \frac 12 \lambda_{\delta}\|\wh {\bm u}_{\delta,\mS^c}\|_1\\
	&\le&
	- \lambda_{\delta}\|\wh {\bm \delta}_{\text{ptl},\mS^c}\|_1 + \frac32 \lambda_{\delta}\|\wh {\bm u}_{\delta,\mS}\|_1 + \frac 12 \lambda_{\delta}\|\wh {\bm u}_{\delta,\mS^c} \|_1\\
	&\le& 
	\frac 32 \lambda_{\delta}\|\wh {\bm u}_{\delta,\mS}\|_1 - \frac 12 \lambda_{\delta} \|\wh {\bm u}_{\delta,\mS^c}\|_1.
	\eeqrs
	As a result, we have $\|\wh {\bm u}_{\delta, \mS^c}\|_1\le 3\|\wh {\bm u}_{\delta,\mS} \|_1 $. Then we have $\|\wh {\bm u}_{\delta}\|_1 = \|\wh {\bm u}_{\delta,\mS}\|_1 + \|\wh {\bm u}_{\delta,\mS^c}\|_1\le 4\|\wh {\bm u}_{\delta,\mS}\|_1 \le 4\sqrt{s_{\delta}}\|\wh {\bm u}_{\delta}\|$. By \eqref{equation rsc for sigmaxx}, we have
	\[
	\wh {\bm u}^\top_{\delta}\wh \Sigma \wh {\bm u}_{\delta} \ge \bigg(\frac{C_{\min}}2 - 16C_{12} \frac{s_{\delta}\log p}n \bigg)\|\wh {\bm u}_{\delta}\|^2.
	\]
	By condition (C2), we have $s_{\delta}\log p/n = o(1)$. Then when $n \ge 32C_{12}C_{\min}^{-1}s_{\delta}\log p$, we have $\|\wh {\bm u}_{\delta}\|^2 \le C_{20}\lambda_{\delta}\|\wh {\bm u}_{\delta,\mS}\|_1\le C_{20}\sqrt {s_{\delta}}\lambda_{\delta}\|\wh {\bm u}_{\delta}\| $ for some constant $C_{20} \ge3 (C_{\min}/2 - 16C_{12}s_{\delta}\log p/n)^{-1}$. As a result, we have $\|\wh {\bm u}_{\delta}\|^2 \le C_{20}^2s_{\delta}\lambda_{\delta}^2$ for some constant $C_{20}>0$.
	
	Then we show that $\|\wh {\bm u}_{\delta}\|^2 \lesssim h^2 $. By triangle inequality, we have
	\begin{gather*}
		0\le \frac 12 \wh {\bm u}^\top_{\delta}\wh \Sigma \wh {\bm u}_{\delta} \le \lambda_{\delta}\|\bm \delta\|_1 - \lambda_{\delta}\|\wh {\bm \delta}_{\text{ptl}}\|_1 + \frac 12\lambda_{\delta}\|\wh {\bm u}_{\delta}\|_1\le 2\lambda_{\delta}\|\bm \delta\|_1 - \frac 12 \lambda_{\delta}\|\wh {\bm u}_{\delta}\|_1.
	\end{gather*}
	Then we know that $\|\wh {\bm u}_{\delta}\|\le \|\wh {\bm u}_{\delta}\|_1 \le 4\|\bm \delta\|_1 \le 4h$.
	
	Lastly, we show that $\|\wh {\bm u}_{\delta}\|^2 \lesssim \lambda_{\delta} h $. By \eqref{equation rsc for sigmaxx} and $ h =O(1)$ in condition (C2), we have
	\beqrs
	\frac{C_{\min}}2 \|\wh {\bm u}_{\delta}\|^2& \le &C_{12}\frac{\log p}{n} \|\wh {\bm u}_{\delta}\|_1^2 + \wh {\bm u}_{\delta}^\top \wh\Sigma \wh {\bm u}_{\delta} \le C_{12}\frac{\log p}{n} \|\wh {\bm u}_{\delta}\|_1^2 + 2\lambda_{\delta}\|\bm \delta\|_1\\
	&\le& C_{12}\frac{h^2 \log p}{n} + 2\lambda_{\delta}h\le C_{21}\lambda_{\delta}h
	\eeqrs
	for some constant $C_{21} \ge C_{12}h \log p/n + 2$.  Combining the above results, we have $\|\wh \delta_{\text{ptl}} - \delta\|^2 = \|\wh u_{\delta}\|^2 \le C_{22}\min\{s_{\delta}\lambda_{\delta}^2, \lambda_{\delta} h, h^2 \}$ for some constants $C_{22}\ge \max\{C_{20}^2, C_{21}, 4\}$.\\
\end{proof}

\csection{Proof of Theorem \ref{Theorem convergence rate}}

%
%

Combining the results of Lemma \ref{Lemma rate w} and \ref{Lemma rate delta}, we have
\beqrs
&&\|\wh{\bm \beta}_{\text{ptl}} - \bm \beta\|^2 = \|\wh B\wh {\bm w}_{\text{ptl}} +\wh{\bm \delta}_{\text{ptl}} -  B\bm w - \bm \delta\|^2\\  &\le &2(\|(\wh B - B)(\wh {\bm w}_{\text{ptl}} - \bm w)\|^2 + \|(\wh B - B)\bm w \|^2 + \|B(\wh {\bm w}_{\text{ptl}} - \bm w)\|^2 + \|\wh {\bm \delta}_{\text{ptl}} -\bm \delta\|^2)\\
&\le& 2(\|\wh B - B\|_F^2 \|\wh {\bm w}_{\text{ptl}} - \bm w\|^2 + \|\bm w\|^2 \|\wh B - B\|_F^2 + C_{\max}^2 \|\wh {\bm w}_{\text{ptl}} - \bm w\|^2 + \|\wh {\bm \delta} - \bm \delta\|^2)\\
&\le& C_3C_{12}^2\Big( \sum_{k=1}^K\frac{s_{\max}\log p}{n_k}\Big) \Big\{ \sum_{k=1}^K\frac{s_{\max}\log p}{n_k} + \frac{K+\log (p\wedge n)}{n}\Big\} \\
&+&\|\bm w\|^2 \Big(C_3\sum_{k=1}^K\frac{s_{\max}\log p}{n_k}\Big) + C_{\max}^2 C_{12}^2\Big\{ \sum_{k=1}^K\frac{s_{\max}\log p}{n_k} + \frac{K+\log(p\wedge n)}{n}\Big\}\\
&+&C_{22}C_{\delta}^2\bigg[\Big\{s_{\delta}\sum_{k=1}^K \frac{s_{\max} \log p}{n_k} + \frac{s_{\delta} (K+\log(p\wedge n))}{n}\Big\}\wedge\\
&&\Big\{h\Big(\sum_{k=1}^K \frac{s_{\max} \log p}{n}\Big)^{1/2} +
h\Big(\frac{K+\log (p\wedge n)}{n}\Big)^{1/2}\Big\}\wedge h^2  \bigg]\\
&\le& C_{23}\Big(r_n^2 + s_{\delta} r_n^2\wedge h r_n\wedge h^2 \Big)
\eeqrs
with probability at least $1-\exp\{-c_1\log(p\wedge n)\}$ for some constant $C_{23} \ge \max\{C_w + C_{\max}^2 C_{12}^2, C_{22}C_{\delta}^2\}$. Then we complete the theorem proof.

\csection{Proof of Corollary \ref{Corollary compare lasso}}

By condition (1) and (2) in Corollary \ref{Corollary compare lasso}, we have $\sum_{k=1}^Ks_{\max}\log p /n_k = o(s_0\log p /n )$ and $K / n = o (s_0 \log p /n )$. As a result, we have $r_n^2 = o(s_0\log p/ n )$. Then it is required that $s_\delta r_n^2 \wedge h r_n \wedge h^2 = o(s_0\log p /n)$. If $h > s_{\delta} r_n$, then $s_\delta r_n^2 \wedge h r_n \wedge h^2 = s_{\delta} r_n^2 = o(s_0\log p /n )$ provided that $s_{\delta} = o(1)$. If $h < s_{\delta} r_n$, then $s_\delta r_n^2 \wedge h r_n \wedge h^2  = h r_n \wedge h^2 = o\{h (s_0\log p /n)^{1/2} \} \wedge h^2 = o(s_0\log p /n )$ provided that $h = o((s_0\log p / n)^{1/2})$. Then by the conclusion in Theorem \ref{Theorem convergence rate}, we have $\|\wh{\bm \beta}_{\text{ptl}} -\bm \beta\|^2 = o_p(s_0\log p / n)$. 

\csection{Proof of Corollary \ref{Corollary compare trans}}

By condition  (1) in Corollary \ref{Corollary compare trans}, we have $\sum_{k=1}^Ks_{\max} \log p /n_k = o(s_0 \log p / n^{\xi})  = o(s_0\log p / (\nu n^\xi + n )) =o(s_0 \log p / (n_{\mathcal A} + n)) $. By condition (2) in Corollary \ref{Corollary compare trans}, we have $K/ n = o (s_0\log p / n^{\xi }) = o(s_0\log p / (n_{\mathcal A} + n))$. As a result, $r_n^2 = o(s_0\log p / (n_{\mathcal A} + n) ) + O(\log p /n) = O(\log p /n )$ provided that $s_0 n / n_{\mathcal A} = O(1)$. If $h > s_{\delta} r_n$, then $h_0 > h > s_{\delta} r_n> (\log p /n )^{1/2}$ and $s_{\delta} r_n^2\wedge h r_n \wedge h^2 = s_{\delta} r_n^2 = O(s_{\delta} \log p /n ) = o(s_0\log p / n\wedge h (\log p /n )^{1/2})$ provided that $s_{\delta} = o(s_0(m_0\wedge 1))$. If $r_n<h < s_{\delta} r_n$, then $h_0 >h > (\log p /n)^{1/2}$ and $s_{\delta} r_n^2 \wedge h r_n \wedge h^2 = h r_n = o\{(s_{0} \log p /n )(1\wedge m_0^{1/2})\} = o\{(s_{0} \log p /n ) \wedge h_0(\log p /n)^{1/2}\}$. If $h <r_n$, then $s_{\delta} r_n^2 \wedge h r_n \wedge h^2 = h^2 = o(s_{0}\log p /n \wedge h_0(\log p/n)^{1/2} \wedge h^2 )$ provided that $h = o\{(s_0\log p/n)^{1/2} (m_0^{1/2}\wedge 1)\}$. 

\csection{Some Useful Lemmas of Theorem \ref{Theorem minimax}}

\begin{lemma}
	(Fano's Inequality) Let $(\Theta, d)$ be a metric space and for each $\theta\in \Theta$, there exists an associated probability measure $P_{\theta}$. Define $M(\eta;\Theta, d) = \{\theta^1,\dots, \theta^M\}$ be an $\eta$-packing set of $\Theta$. That is for all $\theta,\theta'\in M(\eta;\Theta, d)$, we have $d(\theta, \theta' )\ge \eta$. Let $M = |M(\delta;\Theta,d)|$. Then we have
	\[
	\inf_{\wh\theta}\sup_{\theta \in\Theta }P_{\theta} (d(\wh \theta, \theta)\ge \eta/2) \ge 1-\frac{\log 2 + M^{-2}\sum_{j,k=1}^M KL(P_{\theta^j} || P_{\theta^k})}{\log M}
	\]
	where $KL(P||Q)$ is the KL divergence of two probability distributions $P$ and $Q$
	.\label{Lemma Fano}
\end{lemma}
\begin{proof}
	The lemma proof can be found in \cite{Tsybakov2009IntroductionTN} and \cite{wainwright2019high}.
\end{proof}

\begin{lemma}
	(Packing Number of a Discrete Set) Define $\mH(s) = \{z=(z_1,\dots, z_p)^\top \in \{-1,0,1\}^p |\|z\|_0 = s\}$ for $0\le s\le p$ and $\rho_H(z,z') = \sum_{j=1}^p I(z_j\neq z_j')$ to be the Hamming distance. Then for $s < 2p/3$, there exists a subset $\wt \mH(s) \subset\mH(s)$ with cardinality $|\wt \mH(s)|\ge \exp\{2^{-1}s \log\frac{p-s}{s/2} \}$ such that $\rho_H(z,z')\ge s/2$ for all $z, z'\in \wt \mH(s)$. \label{Lemma packing number}
\end{lemma}
\begin{proof}
	The lemma proof can be found in Lemma 4 of \cite{raskutti2011minimax}.
\end{proof}

\begin{lemma}
	(KL Divergence of Linear Regression Model) Consider the linear regression $P_{Y_i|X_i;\bm \beta}$ be the distribution function of $N(X_i^\top \bm \beta, \sigma^2)$. Suppose $X_i$ is independent distributed with covariance matrix $\Sigma$ and the distribution function $P_{X_i}$. Further assume that $\lambda_{\max}(\Sigma)\le C_{\max}$. Denote $P_{\bm \beta} = \prod_{i=1}^n P_{Y_i|X_i;\bm \beta} P_{X_i}$ be the joint distribution of $(X_i,Y_i), i=1,\dots, n$. Then we have $KL(P_{\bm \beta_1}||P_{\bm \beta_2})\le 2^{-1}\sigma^{-2}C_{\max}n\|\bm \beta_1 - \bm \beta_2\|^2$. \label{Lemma KL divergence}
\end{lemma}
\begin{proof}
	Let $\phi(x) = (2\pi)^{-1/2}\exp(-x^2/2)$ be the probability density function of the standard normal distribution. By the additivity of KL divergence, we have
	\beqrs
	KL(P_{\bm \beta_1}||P_{\bm \beta_2})& =& \sum_{i=1}^n KL(P_{Y_i|X_i;\bm \beta_1} P_{X_i}|| P_{Y_i|X_i;\bm \beta_2} P_{X_i}) \\&=& nKL(P_{Y_1|X_1;\bm \beta_1} P_{X_1}||P_{Y_1|X_1;\bm \beta_2} P_{X_1} )\\
	&= &n\int \log\bigg( \frac{\phi\{(Y_i - X_i^\top \bm \beta_1)/\sigma\}}{\phi\{(Y_i - X_i^\top\bm  \beta_2)/\sigma\}}\bigg) dP_{Y_i|X_i;\bm \beta_1} d P_{X_i}\\
	&= &n\sigma^{-2}\int (Y_i - X_i^\top \bm \beta_1) X_i^\top (\bm \beta_1 - \bm \beta_2 ) dP_{Y_i|X_i,\bm \beta_1} d P_{X_i} \\&+& 2^{-1}\sigma^{-2}n\int\{X_i^\top (\bm \beta_1 - \bm \beta_2)\}^2  dP_{Y_i|X_i;\bm \beta_1} d P_{X_i}\\
	&=& 2^{-1}\sigma^{-2}n(\bm \beta_1 - \bm \beta_2)\Sigma (\bm \beta_1 - \bm \beta_2) \\&\le& 2^{-1}\sigma^{-2}C_{\max}n \|\bm \beta_1 - \bm \beta_2\|^2.
	\eeqrs
\end{proof}

\csection{Proof of Theorem \ref{Theorem minimax}}

Define $P_{\bm \beta^{(k)}}^{(k)} = \prod_{i=1}^{n_k}P_{Y_i^{(k)}| X_i^{(k)}; \bm \beta^{(k)}}P_{X_i^{(k)}}$ be the joint distribution of $(X_i^{(k)}, Y_i^{(k)})$ for $1\le k\le K$. Further define $P_B^s = \prod_{k=1}^K P_{\bm \beta^{(k)}}^{(k)}$ be the joint distribution of the source dataset. The theorem can be proved by considering different terms in the lower bound.

\textsc{Step 1.} We first verify the term $\sum_{k=1}^K s_k\log p /(n_k\vee n)$ in the lower bound. For a fixed index $1\le k\le K$, we fix $\{\bm \beta^{(j)}\}_{j\neq k}$, $\bm w$ with $w_k \neq 0$ and $\bm \delta$. To construct a packing set of $\mathcal B_0(s_{\max }) = \{\bm \beta^{(k)}: \|\bm \beta^{(k)}\|_0 \le s_{\max} \}$, we define $\mH(s_{\max}) = \{z = (z_1,\dots, z_p)^\top \in \{-1,0,1\}^p:\|z\|_0= s_{\max} \}$. By Lemma \ref{Lemma packing number},  we can find a subset $\wt \mH(s_{\max} )\subset \mH(s_{\max})$ with cardinality $|\wt \mH(s_{\max})| \ge \exp[2^{-1}s_{\max} \log \{(p-s_{\max}/(s_{\max}/2)\}]$ such that $\rho_H(z,z') \ge s_{\max}/2$ for all $z,z'\in \wt \mH(s_{\max})$. As a result, $\eta \sqrt 2s_k^{-1/2}|w_k|^{-1}\wt \mH(s_{\max})$ is a $|w_k|^{-1}\eta$-packing set for $\mathcal B_0(s_{\max})\cap \mathcal B_2(\sqrt {2}\eta)$, where $\mathcal B_2(R) = \{\bm \beta^{(k)}: \|\bm \beta^{(k)}\|\le R \}$. Then $\eta \sqrt 2s_k^{-1/2}|w_k|^{-1}\wt \mH(s_{\max})$ is also a $|w_k|^{-1}\eta$-packing set for $\mathcal B_0(s_{\max})$. For arbitrary different $\bm \beta^{(k)i} ,\bm \beta^{(k)j} \in \eta\sqrt 2s_{\max}^{-1/2}|w_k|^{-1}\wt \mH(s_{\max}) $, we have $\| \bm \beta^{(k)i } - \bm \beta^{(k)j}\|^2 \le 4s_k \eta^2 2s_{\max}^{-1}w_k^{-2} = 8w_k^{-2}\eta^2$. Let $\bm \beta^{k,i}= B^{k,i} \bm w + \bm \delta$ with $B^{k,i} = (\bm \beta^{(1)},\dots, \bm \beta^{(k)i},\dots,\bm \beta^{(K)})$ for $i=1,\dots, M$ with $M = |\wt \mH(s_{\max})|$. Then $\|\bm \beta^{k,i} - \bm \beta^{k,j}\| = \|(B^{k,i} - B^{k,j}) \bm w\| = |w_k| \|\bm \beta^{(k)i } - \bm \beta^{(k)j} \| \ge \eta$.  Then by Lemma \ref{Lemma KL divergence} and the additivity of KL divergence, we have
\beqrs
KL(P_{B^{k,i}}^sP_{\bm \beta^{k,i}}||P_{B^{k,j}}^sP_{\bm \beta^{k,j}})  &=& KL(P_{B^{k,i}}^s|| P_{B^{k,j}}^s) + KL(P_{\bm \beta^{k,i}}|| P_{\bm \beta^{k,j}}) \\&=& KL(P_{\bm \beta^{(k)i}}^{(k)} || P_{\bm \beta^{(k)j}}^{(k)}) + KL(P_{\bm \beta^{k,i}}|| P_{\bm \beta^{k,j}}) \\
&\le& \frac{n_k}{2\sigma_k^2}C_{\max}\|\bm \beta^{(k)i} - \bm \beta^{(k)j}\|^2 + \frac{n}{2\sigma^2}C_{\max }\|\bm \beta^{k,i} - \bm \beta^{k,j} \|^2 \\&\le& C_{\max} \Big(\frac{n_k}{2\sigma_k^2} + \frac{n}{2\sigma^2}|w_k|^2\Big) \|\bm \beta^{(k)i} - \bm \beta^{(k)j}\|^2 \\&\le& 4C_{\max}\Big(\frac{n_k}{|w_k|^2\sigma_k^2} + \frac{n }{\sigma^2}\Big) \eta^2.
\eeqrs
Take $\eta = C_{24}\{s_{\max}\log p/(n_k\vee n)\}^{1/2}$ for some small enough $C_{24} >0$. Then by Fano's inequality (Lemma \ref{Lemma Fano}), we have
\beqrs
& &\inf_{\wh{\bm \beta}} \sup_{\bm \beta\in \Theta(\bm s,h)} P\Big(\|\wh{\bm \beta}- \bm \beta\|^2
\ge  C_{24} \frac{s_{\max}\log p}{n_k\vee n} \Big) \\ &\ge& 1- \frac{\log 2 + 2C_{24}^2(w_k^{-2} \sigma_k^{-2}+ \sigma^{-2})C_{\max} s_{\max}\log p }{2^{-1}s_k\log \{(p-s_{\max})/(s_{\max}/2)\}} \ge \frac 12.
\eeqrs
By varying the index $k$, we have
\[
\inf_{\wh{\bm \beta}} \sup_{\bm \beta\in \Theta(s,h)} P\Big(\|\wh{\bm \beta}- \bm \beta\|^2 \ge C_{25} \sum_{k=1}^K \frac{s_{\max}\log p}{n_k\vee n} \Big) \ge \frac 12
\]
for some positive constant $C_{25}>0$.

\textsc{Step 2. }Then we verify the $K/n$ in the lower bound. We apply the local packing technique to verify the conclusion \citep{wainwright2019high}. We fix $\{\bm \beta^{(k)}\}_{k=1}^K$ and $\bm \delta$ satisfying $C_{\min}\le \lambda_{\min}(B^\top B)\le \lambda_{\max}(B^\top B)\le C_{\max}$. For any $\eta >0$, denote $\{w^1,\dots, w^M\}$ is a maximal $C_{\min}^{-1/2}\eta$-packing set of the set $\Theta_w = \{\bm w:\|\bm w\| \le2C_{\min}^{-1/2} \eta \}$. Then by Corollary 4.2.13 of \cite{vershynin2018high}, we have $\log M \ge \log N(C_{\min}^{-1/2}\eta;\Theta_w, \|\cdot \|)\ge K\log 2$, where $N(C_{\min}^{-1/2}\eta; \Theta_w, \|\cdot\|)$ is the $C_{\min}^{-1/2}\eta$-covering number of $\Theta_w$ under the $\ell_2$ norm. As a result, we have $C_{\min}^{-1/2}\eta \le \|\bm w^i - \bm w^j \|\le 4C_{\min}^{-1/2}\eta$. Denote $\bm\beta^i = B\bm w^i + \bm \delta$ for $i = 1,\dots, M$. Then we have $\|\bm \beta^i - \bm \beta^j\|^2 = \|B(\bm w^i - \bm w^j)\|^2\ge C_{\min }\|\bm w^i - \bm w^j\|^2\ge \eta^2$ and $\|\bm \beta^i - \bm \beta^j\|^2 =\|B(\bm w^i - \bm w^j)\|^2 \le C_{\max}\|\bm w^i -\bm  w^j \|^2 \le C_{\max}C_{\min}^{-1}\eta^2$. Then by Lemma \ref{Lemma KL divergence}, we have $KL(P_{B}^sP_{\bm \beta^i}||P_{B}^sP_{\bm \beta^j}) = KL(P_{\bm \beta^i}|| P_{\bm \beta^j}) \le2^{-1}\sigma^{-2}C_{\max}^2C_{\min}^{-1}n\eta^2$. Let $\eta = C_{26}\sqrt{K/n}$ and by applying Fano's inequality, we have
\[
\inf_{\wh{\bm \beta}} \sup_{\bm\beta\in \Theta(\bm s, h)} P\Big(\|\wh{\bm \beta}- \bm \beta\|^2 \ge C_{26}\frac Kn \Big) \ge 1- \frac{\log 2 + 2^{-1}\sigma^{-2}C_{\max}^2 C_{\min}^{-1}C_{26}^2 K }{K\log 2} \ge \frac 12
\]
for some small enough constant $C_{26} >0$.

\textsc{Step 3.} Then we verify the term $(s_{\delta }\log p/n)\wedge h(\log p/n)^{1/2}\wedge h^2$. We fix $B$ with $\beta^{(k)} = e_k$ and $w$, where $\{e_j\}_{j=1}^p$ forms the canonical basis of $\mR^p$. We also fix the covariance matrix of $X_i$ to be $I_p$. Then the identification condition $\bm \beta^{(k)\top} \Sigma\bm \delta = 0$ becomes $\delta_i = 0,i=1,\dots,K$. Then we modify the proof from \cite{rigollet2011exponential} to prove the lower bound. Denote $m = \lfloor h(\log p/n)^{-1/2}\rfloor$, where $\lfloor x \rfloor$ is the largest integer no greater than $x$. Then the conclusion holds by considering the following three steps. In Step 3.1, we prove the lower bound when $(s_{\delta }\log p/n)\wedge h(\log p/n)^{1/2}\wedge h^2 = h^2$. In Step 3.2, we consider the case when $(s_{\delta }\log p/n)\wedge h(\log p/n)^{1/2}\wedge h^2 = h(\log p/n)^{1/2}$. In Step 3.3, the situation becomes $(s_{\delta }\log p/n)\wedge h(\log p/n)^{1/2}\wedge h^2 = s_{\delta }\log p/n$.

\textit{Step 3.1.} If $m = 0$, then we have $h^2<\log p/n$. Define $\mH_K(s_{\delta}) = \{z = (z_1,\dots, z_p)^\top\in \mR^p: z\in \mH(s_{\delta}), z_i = 0,i=1,\dots,K\}$. Note that the conclusion of Lemma \ref{Lemma packing number} can be also apply to parameters with common sub-vector. Then by Lemma \ref{Lemma packing number}, we can find a subset $\wt \mH_{K}(s_{\delta}) \subset \mH_K(s_\delta)$ such that $|\wt \mH_K(s_{\delta})| \ge \exp[2^{-1}s_{\delta}\log \{(p-K-s_{\delta})/(s_{\delta}/2)\}]$ and $\rho_H(z,z')\ge s_{\delta}/2$ for all $z,z'\in \wt \mH_K(s_{\delta})$ with $z\neq z'$. Then for $\delta^i\neq \delta^j\in C_{27} h \wt \mH_K(1)$, we have $\|\bm \delta^i - \bm \delta^j \| \ge C_{27}h$ and $\|\bm \delta\|_1\le C_{27}h \le h$ for $C_{27}<1$. Define $\bm \beta^i = B\bm w + \bm \delta^i$ for $i = 1,\dots, |\wt\mH_K(1)|$. Then by the above arguments, we know that $\bm \beta^i \in \Theta(\bm s,h)$ and $\|\bm \beta^i - \bm \beta^j \| = \|\bm \delta^i -\bm \delta^j\| \ge C_{27}h$. In the meanwhile, by Lemma \ref{Lemma KL divergence}, we have $KL(P_{B}^sP_{\bm \beta^i}|| P_B^s P_{\bm \beta^j}) = KL(P_{\bm \beta^i}||P_{\bm \beta^j}) \le 2^{-1}\sigma^{-2}C_{\max}n\|\bm \beta^i - \bm \beta^j\|^2=2^{-1}\sigma^{-2}C_{\max}n\|\bm \delta^i - \bm \delta^j\|^2\le \sigma^{-1}C_{\max}C_{27}^2nh^2 $. Then by Fano's inequality (Lemma \ref{Lemma Fano}), we have
\[
\inf_{\wh{\bm \beta}}\sup_{\bm \beta\in \Theta(\bm s,h)} P(\|\wh{\bm \beta} - \bm \beta\| \ge C_{27}h)\ge 1- \frac{\log 2 + 2^{-1}\sigma^{-2} C_{\max}C_{27}^2n h^2}{\log (p - K)}\ge \frac 12.
\]

\textit{Step 3.2.} If $m\ge 1$ and $s_{\delta}\log p/n > h(\log p/n)^{1/2}$ then we have $s_{\delta} > m$ and $h > (\log p/n)^{1/2}$. By a similar step in Case 1, we know that there exists $\wt \mH_K(m)$ such that $|\wt \mH_K(m) |\ge \exp[2^{-1}m\log \{(p-K-m)/(m/2)\}]$ and $\rho_H(z,z') \ge m/2$ for all $z,z'\in \wt \mH_K(m)$ with $z\neq z'$. Then for $\bm \delta^i\neq \bm \delta^j \in \sqrt{2}m^{-1/2}\eta \wt \mH_K(m) $, we have $\|\bm \delta^i - \bm \delta^j\|^2 \le 4m^{-1} \eta^2 2m = 8\eta^2$ and $\|\bm \delta^i - \bm \delta^j\|^2 \ge\eta^2$. In the meanwhile, we have $\|\bm \delta^i\|_1 = (2m)^{1/2}\eta \le h$ for $\eta = C_{28}h^{1/2}(\log p/n)^{1/4}$ with some small enough constant $C_{28} < 1/\sqrt 2$. Define $\bm \beta^i = B\bm w + \bm \delta^i$ for $i = 1,\dots, |\wt \mH_K(m)|$. Then by the above arguments, we have $\bm \beta^i \in \Theta(\bm s,h)$ and $\|\bm \beta^i - \bm \beta^j \| = \|\bm \delta^i - \bm \delta^j\|\ge \eta$. By Lemma \ref{Lemma KL divergence}, we have $KL(P_{B}^sP_{\bm \beta^i}|| P_B^s P_{\bm \beta^j}) =KL(P_{\bm \beta^i} || P_{\bm \beta^j}) \le 2^{-1}\sigma^{-2}C_{\max} n\|\bm \beta^i - \bm \beta^j \|^2 =2^{-1}\sigma^{-2}C_{\max} n\|\bm \delta^i - \bm \delta^j \|^2\le 4\sigma^{-2}C_{\max}C_{28}^2h(n\log p)^{1/2}$. Recall that $m = \lfloor h(\log p/n)^{-1/2}\rfloor$. As a result, by Fano's inequality (Lemma \ref{Lemma Fano}), we have
\beqrs
& &\inf_{\wh{\bm \beta}}\sup_{\bm \beta\in \Theta(\bm s,h)} P\bigg(\|\wh{\bm \beta} - \bm \beta\|^2 \ge C_{28}h\Big(\frac{\log p}{n}\Big)^{1/2}\bigg)\\
&\ge& 1- \frac{\log 2 + 2^{-1}\sigma^{-2} C_{\max}C_{28}^2h(n\log p)^{1/2}}{2^{-1}m\log \{(p-K-m)/(m/2)\}}\ge \frac 12,
\eeqrs
when $C_{28}$ is sufficiently small.

\textit{Step 3.3.} If $m\ge 1$ and $s_{\delta} \log p/n < h(\log p/n)^{1/2}$, then we have $h > s_{\delta}(\log p/n)^{1/2}$. By a similar step in Case 1, we know that there exists $\wt \mH_K(s_{\delta})$ such that $|\wt \mH_K(s_{\delta}) |\ge \exp[2^{-1}s_{\delta}\log \{(p-K-s_{\delta})/(s_{\delta}/2)\}]$ and $\rho_H(z,z') \ge s_{\delta}/2$ for all $z,z'\in \wt \mH_K(s_{\delta})$ with $z\neq z'$. Then for $\bm \delta^i\neq \bm \delta^j \in \sqrt{2}s_{\delta}^{-1/2}\eta \wt \mH_K(s_{\delta}) $, we have $\|\bm \delta^i - \bm \delta^j\|^2 \le 4s_{\delta}^{-1} \eta^2 2s_{\delta} = 8\eta^2$ and $\|\bm \delta^i - \bm \delta^j\|^2 \ge\eta^2$. In the meanwhile, we have $\|\bm \delta^i\|_1 = (2s_{\delta})^{1/2}\eta \le h$ for $\eta = C_{29}(s_{\delta}\log p/n)^{1/2}$ with some small enough constant $C_{29} < 1/\sqrt 2$. Define $\bm \beta^i = B\bm w + \bm \delta^i$ for $i = 1,\dots, |\wt \mH_K(s_{\delta})|$. Then by the above arguments, we have $\bm \beta^i \in \Theta(\bm s,h)$ and $\|\bm \beta^i - \bm \beta^j \| = \|\bm \delta^i - \bm \delta^j\|\ge \eta$. By Lemma \ref{Lemma KL divergence}, we have $KL(P_{B}^sP_{\bm \beta^i}|| P_B^s P_{\bm \beta^j}) =KL(P_{\bm \beta^i} || P_{\bm \beta^j}) \le 2^{-1}\sigma^{-2}C_{\max} n\|\bm \beta^i - \bm \beta^j \|^2 =2^{-1}\sigma^{-2}C_{\max} n\|\bm \delta^i - \bm \delta^j \|^2\le 4\sigma^{-2}C_{\max}C_{29}^2s_{\delta}\log p$. As a result, by Fano's inequality (Lemma \ref{Lemma Fano}), we have
\[
\inf_{\wh{\bm \beta}}\sup_{\bm \beta\in \Theta(\bm s,h)} P\bigg(\|\wh{\bm \beta} - \bm \beta\|^2 \ge C_{29}\frac{s_{\delta}\log p}{n}\bigg)\ge 1- \frac{\log 2 + 2^{-1}\sigma^{-2} C_{\max}C_{29}^2s_{\delta}\log p}{2^{-1}s_{\delta}\log \{(p-K-s_{\delta})/(s_{\delta}/2)\}}\ge \frac 12,
\]
when $C_{29}$ is sufficiently small.
\end{CJK}
\end{document}